\documentclass[11pt]{amsart}
\usepackage[
  hmarginratio={1:1},     % equal left and right margins
  vmarginratio={1:1},     % equal top and bottom margins
  textwidth=410pt,        % new text width
  heightrounded,          % always useful
]{geometry}          % ... or a4paper or a5paper or ... 
\usepackage{graphicx}
\usepackage{amssymb, xcolor}
\usepackage{enumerate}
\usepackage{amsfonts,amssymb,amsthm}
\usepackage{datetime}
\mmddyyyydate
\settimeformat{hhmmsstime}

\newtheorem{theorem}{Theorem}[section]

\newtheorem{lem}[theorem]{Lemma}
\newtheorem{prop}[theorem]{Proposition}
\newtheorem{thm}[theorem]{Theorem}

\theoremstyle{remark}
\newtheorem{rem}{Remark}
\usepackage{hyperref}
\makeatletter
\let\reftagform@=\tagform@
\def\tagform@#1{\maketag@@@{(\ignorespaces\textcolor{purple}{#1}\unskip\@@italiccorr)}}
\renewcommand{\eqref}[1]{\textup{\reftagform@{\ref{#1}}}}
\makeatother

\hypersetup{colorlinks=true,pdfborder=000,  citecolor  = blue, citebordercolor= {purple}}
\hypersetup{linkcolor=purple}

\usepackage{mathrsfs}

\newcommand{\E}{\ensuremath{\mathbb{E}}}
\newcommand{\Pro}{\ensuremath{\mathbb{P}}}
\newcommand{\Tr}{\textup{Tr}}
\newcommand{\slvar}{\stackrel{\textup{sl.}}{\sim}}

\newcommand{\mathdef}{\mathrel{\mathop:}=}

\newcommand{\br}{\vspace{1mm}\\}
\newcommand{\Br}{\vspace{2mm}\\}

\def\Var{\mathop{\rm Var}\nolimits}

\def\<{\langle}
\def\>{\rangle}

\title[Sparse, heavy tailed random matrices]{Extreme eigenvalues of sparse, heavy tailed random matrices}

\begin{document}

\author{Antonio Auffinger}
\address{Antonio Auffinger - Northwestern University, Department of Mathematics, 2033 Sheridan Road, Evanston, IL 60208, USA.}

\email{tuca@northwestern.edu}
\thanks{The research of A. A. is supported by NSF grant DMS-1517864.}
\author{Si Tang}
\address{Si Tang - University of Chicago, Department of Statistics, 5734 S. University Avenue, Chicago, IL 60637, USA.}
 \email{sitang@galton.uchicago.edu}
\begin{abstract} We study the statistics of the largest eigenvalues of $p \times p$ sample covariance matrices $\Sigma_{p,n} = M_{p,n}M_{p,n}^{*}$ when the entries of the $p \times n$ matrix $M_{p,n}$ are sparse and have a distribution with tail $t^{-\alpha}$, $\alpha>0$. On average the number of nonzero entries of $M_{p,n}$ is of order $n^{\mu+1}$, $0 \leq \mu \leq 1$. We prove that in the large $n$ limit, the largest eigenvalues are Poissonian if $\alpha<2(1+\mu^{{-1}})$ and converge to a constant in the case $\alpha>2(1+\mu^{{-1}})$. We also extend the results of \cite{BenaychGeorges:2014ki} in the Hermitian case, removing restrictions on the number of nonzero entries of the matrix. 
\end{abstract}
\maketitle


\section{Introduction}

We study the statistics of the largest eigenvalues of sample covariance matrices when the entries are heavy tailed and sparse. Let $x$ be a complex-valued random variable. We say $x$ has a heavy tailed distribution with parameter $\alpha$ if the (two-sided) tail probability
\begin{equation*}
\label{eqn:distalpha}
G_{\alpha}(t) \mathdef \Pro(|x| > t) = L(t) t^{-\alpha},\quad t>0
\end{equation*}
where $\alpha >0$ and $L$ is a slowly varying function, i.e.,
\begin{equation*}
 \lim_{t\to \infty} \frac{L(st)}{L(t)} =1, \quad \forall s>0.
\end{equation*}

For each $n \ge 1$, let $y = y(n)$ be a Bernoulli random variable, independent of $x$, with $\Pro(y=1) =n^{\mu-1} = 1-\Pro(y=0)$, where $0 \leq \mu \le 1$ is a constant. The ensemble of random sample covariance matrices that we study here is defined as follows. 
For each $n\ge 1$, let $p=p(n)\in \mathbb Z_{+}$ be a function of $n$ such that $$p/n \to \rho, \quad 0 < \rho \le 1,$$ as $n\to \infty$.   Let $A_{p, n}=[a_{ij}]_{i,j=1}^{p,n}$ and  $B_{p, n}=[b_{ij}]_{i,j=1}^{p,n}$ be $p\times n$ random matrices whose entries are i.i.d. copies of $x$ and $y$, respectively. Form the $p\times n$ matrix $M_{p,n}=A_{p,n}\cdot B_{p,n}=[m_{ij}]_{i,j=1}^{p,n}$ by setting $m_{ij}=a_{ij}b_{ij}$. Then 
$$\Sigma_{p,n} \mathdef  M_{p,n}M_{p,n}^{*}$$
 is the $p\times p$ sparse, heavy tailed random sample covariance matrix with parameters $\alpha$ and $\mu$. Note that $\Sigma_{p,n}$ is positive semi-definite so all its eigenvalues are non-negative.

The extreme eigenvalues of $\Sigma_{p,n}$ are the main subject of this paper. We will see that, depending on the tail exponent $\alpha$ and the sparsity exponent $\mu$, when properly rescaled, the top eigenvalues will either converge to a Poisson point process or to the right edge of the Marchenko-Pastur law.

To put our theorems in context, we briefly review past results. The study of extreme eigenvalues of heavy tailed random matrices started with the work of Soshnikov. In \cite{Soshnikov:2004uc}, he proved that if $0<\alpha<2$, the asymptotic behavior of the top eigenvalues of a heavy tailed Hermitian matrix is determined by the behavior of the largest entries of the matrix, i.e.,  the point process of the largest eigenvalues (properly normalized) converges to a Poisson point process, as in the usual extreme value theory for i.i.d. random variables. 
This result was extended to sample covariance matrices and for all values of $\alpha \in (0,4)$  in the work of Auffinger, Ben Arous and P\'{e}ch\'{e} \cite{Auffinger:2009vs}. The upper bound on the tail exponent $\alpha$ is optimal as for entries with finite fourth moment, the largest eigenvalues converge to the right edge of the bulk distribution and have Tracy-Widom fluctuations \cite{Bai1988166,Bai:1988je,lee2014,Yin:1988kp}. Eigenvector localization and delocalization were studied in \cite{BenaychGeorges:2014gg}. In the physics literature, many of these results were predicted in the seminal paper of Bouchaud and Cizeau \cite{CB94}.

The largest eigenvalues of sparse Hermitian random matrices with bounded moments were investigated by Benaych-Georges and P\'ech\'e \cite{ECP3027} under the assumptions of at least $\omega(\log n)$ nonzero entries in each row. They extended the results of \cite{Khor08,Sodin}, establishing the convergence of the largest eigenvalue to the edge and  also obtained results on  localization/delocalization of eigenvectors. For bulk statistics in the sparse setting, readers are invited to see Erd{\H o}s, Knowles, Yau, and Yin \cite{Erdos:2012cx} and the references therein.

In  \cite{BenaychGeorges:2014ki}, Benaych-Georges and P\'ech\'e considered a class of $n \times n$ Hermitian, heavy tailed, sparse matrices. In their work, the authors looked at matrices, where in $n-o(n)$ rows, the number of nonzero entries was asymptotically equal to $n^{\mu}$ for $\mu \in (0,1]$. For the remaining $o(n)$ rows, the number of nonzero entries was no more than $n^{\mu}$. This assumption is well-suited to treat the case of heavy-tailed band matrices.  In the last section, we will extend the work of \cite{BenaychGeorges:2014ki} by removing all restrictions on the number of nonzero entries in each row, allowing, for instance, the sparsity to come from the adjacency matrix of an Erd\H os-R\' enyi random graph. 

Although we extend the results of \cite{BenaychGeorges:2014ki}, the main objective of this paper is to treat the spectrum of sample covariance matrices $\Sigma_{p,n}$ constructed from a sparse matrix $M_{p,n}$.  These matrices naturally appear in applications such as models of complex networks with two species of nodes \cite{Nagao13} and also in information theory as channel capacity of wideband CDMA schemes \cite{Yoshida}. For more applications and predictions one can look at \cite{540625382,1751-8121-40-19-003,540625391,SemerC02} and the references therein. In the mathematical literature, as far as we know, there are no results  dealing with the top eigenvalues of sparse sample covariance matrices. The main purpose of this paper is to provide such results.

Throughout the paper, we will use $\lambda_{l}(A)$ to denote the $l$-th largest eigenvalue of a Hermitian matrix $A$, $\mathbf v_{l}(A)$ the corresponding eigenvector. For a matrix $A = [a_{ij}]$, either Hermitian or rectangular, $a_{i_{l}j_{l}}$ denotes its the $l$-th largest entry in absolute value in the upper-triangular part (if  $A$ is Hermitian) or of all entries (if $A$ is rectangular), and $\theta_{l}(A)$ be its argument, i.e., $\theta_{l}(A) = \arg(a_{i_{l}j_{l}})$. Let $\mathbf e_{1}, \ldots, \mathbf e_{n} $ represent the canonical basis vectors for $\mathbb R^{n}$. The notation $f(x)\slvar g(x)$ means that there exists some slowly varying function $l(x)$ such that $f(x)=l(x)g(x)$.  A sequence of events $(E_n)_{n\geq 1}$ is said to occur {\em with exponentially high probability} (w.e.h.p.) if there exists $C, \theta > 0$ and $n_0 \in \mathbb N$ such that for $n\geq n_0$, $\Pro(
E_{n}) \ge 1-e^{-Cn^{\theta}}$. 
We will also use the following matrix norms:
$$||A||_{\infty} \mathdef \max_{i} \sum_{j} |a_{ij}|, \quad ||A||_{1} \mathdef \max_{j} \sum_{i} |a_{ij}|, \quad ||A|| \mathdef \max_{\mathbf v: ||\mathbf v||_{2}=1} ||A\mathbf v||_{2}.$$

The rest of the sections will be organized as follows. In Section \ref{sec2}, we state our main results. In Section \ref{sec3}, a few key lemmas will be listed and proved. Section \ref{sec4} will be devoted to the proof of the main theorems while in Section \ref{sec:extension} we present the Hermitian case and other extensions.
\section{Main results}\label{sec2}

Our main results are about the eigenvalues of the sample covariance matrix $\Sigma_{p,n}=M_{p,n}M_{p,n}^{*}$. In our setting, there are approximately $p\cdot n\cdot n^{\mu-1}\approx\rho n^{1+\mu}$ nonzero entries in $M_{p,n}$. We know from extreme value theory \cite[Section 1.2]{Resnick:2007uq} that the scaling factor for the largest entries of the matrix $M_{p,n}$ should be
\begin{align} 
\label{eqn:cnp}
c_{np}\mathdef \inf\left\{t: G_{\alpha}(t) \le \frac{1}{pn^{\mu}}\right\}.
\end{align}
Moreover, $c_{np}\slvar n^{(1+\mu )/\alpha}$ and
\begin{align*}
\lim_{n\to \infty} \Pro\left(\max_{i,j}\ c^{-1}_{np}|m_{ij}| \le t\right) = e^{-t^{-\alpha}}.
\end{align*}

Our first theorem says that when $0 < \alpha < 2(1+\mu^{-1})$, the extreme eigenvalues of $\Sigma_{p,n}$ behave like the square of the top entries of $M_{p,n}$.
\begin{theorem}\label{thm:mainresult3} 
Suppose $0 < \alpha < 2(1+\mu^{-1})$. For $(1+\mu^{-1}) \le \alpha < 2(1+\mu^{-1})$, we also assume that $x$ is centered. Then as $n\to \infty$, we have for each $l\geq 1$
\[
\frac{\lambda_{l}(\Sigma_{p,n})}{|m_{i_{l}j_{l}}|^{2}} \stackrel{P}{\rightarrow} 1.
\]
The eigenvectors are localized: for each $l \ge 1$,
$\displaystyle
\left \Vert \mathbf v_{l}(\Sigma_{p,n}) -\mathbf e_{i_{l}}\right \Vert_{2} \stackrel{P}{\longrightarrow} 0.
$
\end{theorem}
It follows from Theorem \ref{thm:mainresult3} and a routine computation (see \cite[Page 593]{Auffinger:2009vs}) that the random point processes
\begin{equation} 
\label{eqn:randppcov1}
\mathcal Q_{n} =\sum_{i=1}^{p}\sum_{j=1}^{n} \delta_{c_{np}^{-2}|m_{ij}|^{2}}, \quad\hat{\mathcal Q}_{n} =\sum_{i=1}^{p} \delta_{c_{np}^{-2}\lambda_{i}(\Sigma_{p,n})}
\end{equation}
converge in distribution to the same Poisson point process  on $(0, + \infty)$ with intensity~$\frac{\alpha}{2x^{1+\alpha/2}}$. 

\begin{rem}
As mentioned in the introduction, the conclusion of the Theorem above holds in the non-sparse case if and only if $0<\alpha<4$. Roughly speaking, when we introduce sparseness, we increase the localization of the eigenvectors towards the position of the largest entry and the Poissonian limit holds with lighter tails ($2(1 + \mu^{-1})>4$). Note that, when $\mu=0$, any polynomial tail is allowed.  This was also observed in \cite{BenaychGeorges:2014ki}, see Section \ref{sec:extension} below for more details. One should also note that although $M_{p,n}$ is sparse, $\Sigma_{p,n}$ is, in general, not. 
\end{rem}

In the second regime, $\alpha > 2(1 + \mu^{-1})$, the Poissonian limit no longer holds. The largest eigenvalues, when normalized by $n^{\mu}$, converge to the edge of the bulk distribution. We also need the following definition. For $L\in \mathbb N$ and $\eta \in (0,1]$, we say that a unit vector $\mathbf v=(v_{1}, \ldots, v_{n})\in \mathbb C^n$ is $(L,\eta)$-localized if there exists a set $S \subseteq \{1,\ldots n \}$ with cardinality $L$ such that $$\sum_{j\in S} |v_j|^2 > 1-\eta.$$ 
\begin{theorem} \label{thm:mainresult4} 
Suppose $\alpha > 2(1+\mu^{-1})$ and $x$ has mean zero and variance one. 
Then for each $l\ge 1$, as $n\to \infty$, we have
\begin{equation*}\label{eq:supercriticalSC}
\frac{\lambda_{l}(\Sigma_{p,n})}{n^{\mu}} \stackrel{P}{\longrightarrow} (1 + \sqrt{\rho})^2.
\end{equation*}
The eigenvectors of $\Sigma_{p,n}$ are delocalized, namely, there exists $\beta, \eta_0>0$ such that for each $l\geq 1$, $0<\eta < \eta_0$  
we have
\begin{equation}\label{eq:delocthmcov}
\mathbb P \left( \mathbf{v}_l(\Sigma_{p,n}) \text{ is } (\lfloor p^{\beta}\rfloor,\eta)\text{-localized} \right) \to 0
\end{equation}
as $n\to \infty$.
\begin{rem} In the regime of both Theorems  \ref{thm:mainresult4} and  \ref{thm:mainresult2} below, when $\mu = 0$, the critical case of a Erd\H os-R\' enyi adjacency matrix, we are forced to take $\alpha = \infty$, which is not allowed. In this case, it is still an open question to obtain explicit formulas for the limiting spectral distribution. To see more in this direction, the reader is invited  to check \cite{Erdos:2012cx} and the references therein.
\end{rem}

\end{theorem}

\begin{rem} The form of delocalization in \eqref{eq:delocthmcov} is relatively simple compared to the results obtained in \cite{BenaychGeorges:2014gg, Bourgade:2013wd} when considering non-heavy tailed distributions for Wigner matrices. In words, \eqref{eq:delocthmcov} says that if $\alpha > 2(1+\mu^{-1})$ eigenvectors must have nonzero coordinates spread over at least $p^\beta$ coordinates, different from the case $\alpha < 2(1+\mu^{-1})$ where the number of nonzero coordinates does not diverge with $n$. 
\end{rem}

\begin{rem}
One can also take  $\Pro (y=1)=f(n)n^{\mu-1}$ for a slowly varying function $f \ne 0$. The results of the above theorems remain true, with an additional slowly varying function in the normalization of the entries. 
\end{rem}

\begin{rem}\label{rem:noname} Our results also hold if we deterministically specify the positions of the nonzero entries in $M_{p,n}$ or if the number of nonzero entries in each row is nonrandom and $\slvar n^{\mu}$. See Remark \ref{bandremark} in Section \ref{sec3}.
\end{rem}

\section{Some useful lemmas}\label{sec3}

In this section we collect some tools and lemmas that will be used throughout the proofs of the main results. 

\subsection{Results on the magnitudes of entries.}

\begin{lem}\label{lem:keylemma1} 

Suppose $M_{p,n}$ is the $p\times n$ rectangular, sparse, heavy tailed matrix. Let $c_{np}$ be as given in (\ref{eqn:cnp}). Then, for all values of $\alpha > 0$ and  $\eta > 0$, we have:
\begin{enumerate}[\normalfont (a)]
\item \label{lemma3.1a} $ \Pro \left\{\exists i, \exists j_{1}\ne j_{2}, 1 \le i \le p, 1\le j_{1}, j_{2} \le n: \min (|m_{ij_{1}}|, |m_{ij_{2}}|)  > c_{np}^{\frac{1+2\mu}{2+2\mu}+\eta} \right\} \to 0
$.
\item  \label{lemma3.1b} $\Pro \left\{\exists j, \exists i_{1}\ne i_{2}, 1\le i_{1}, i_{2} \le p, 1 \le j \le n:   \min (|m_{i_{1}j}|, |m_{i_{2}j}|) > c_{np}^{\frac{1+2\mu}{2+2\mu}+\eta} \right\} \to 0
$.
\end{enumerate}
\end{lem}
\begin{proof} Since $a_{ij}$ has a heavy tailed distribution as given in (\ref{eqn:distalpha}), then
\[
\Pro(|a_{ij}| > t^{\theta}) \slvar t^{-\alpha \theta}, \quad \forall \theta > 0.
\]
Hence for the sparse matrix $M_{p,n}$, we have
\[
\Pro(|m_{ij}| > c_{np}^{\theta}) \slvar n^{\mu-1} c_{np}^{-\alpha \theta} \slvar n^{\mu-1} n^{-(1+\mu )\theta}.
\]
The proof for will  follow from this fact and a union bound. Precisely, the left side of \eqref{lemma3.1a} is bounded by  
\begin{align*}
pn^{2} \left( n^{\mu-1}n^{-(1+\mu)(\frac{1+2\mu}{2+2\mu}+{\eta})} \right)^{2} = pn^{2\mu}n^{-(1+\mu)(\frac{1+2\mu}{1+\mu}+2\eta)} = \frac{p}{n}\cdot n^{-2(1+\mu)\eta} \to 0.
\end{align*}
where we used the fact that $p/n\to \rho$. The proof of \eqref{lemma3.1b} is similar.

\end{proof}

\subsection{Results on the sum of entries within rows and columns}
The following lemma will be used to control the sum of absolute values within a given row or a given column of $M_{p,n}$.
\begin{lem}\label{lem:keylemma2} 
Let $M_{p,n}$ be the sparse, heavy tailed, rectangular random matrix. 
\begin{enumerate}[\normalfont (a)]
\item \label{lemma3.2a} For any sequence $\beta_{n}\slvar n^{b}$, where $0 \le b \le\frac{\mu}{\alpha}$,  and $\forall \epsilon > 0$, then w.e.h.p.,
\begin{equation}\label{eq:lemHiSi}
\sum_{j=1}^{n} |m_{ij}| \mathbf 1_{\{ |m_{ij}| < \beta_{n}\}} \le n^{\mu + b(1-\alpha)^{+} + \epsilon},
\end{equation}
where $(1-\alpha)^{+} \mathdef \max\{1-\alpha,\ 0\}$.
\item \label{lemma3.2b} If $\alpha > 1$ and $\mu >0$, then for any sequences $\alpha_{n}\slvar n^{a}$ and $\beta_{n} \slvar n^{b}$ with $0~\le~a~<~b~\le~\frac{\mu}{\alpha}$, and for any $\epsilon > 0$, we have, w.e.h.p.,
\[
\sum_{j=1}^{n}|m_{ij}|\mathbf 1_{\{\alpha_{n} < |m_{ij}| \le \beta_{n}\}} \le n^{\mu - (\alpha-1)a + \epsilon}.
\]
\item \label{lemma3.2c} If $\mu >0$, then for any   sequences $\alpha_{n}\slvar n^{\frac{\mu}{\alpha} - \eta}$ and $\beta_{n} \slvar n^{\frac{\mu}{\alpha}+\eta'}$ with $\eta, 
\eta' \ge 0$, and any $\epsilon > \alpha \eta + \eta'$, we have, w.e.h.p.,
\[
\sum_{j=1}^{n}|m_{ij}|\mathbf 1_{\{\alpha_{n} < |m_{ij}| \le \beta_{n}\}}\le n^{\frac{\mu}{\alpha}+ \epsilon}. 
\]
\item \label{lemma3.2d} For any   sequences $\alpha_{n}\slvar n^{(\mu/\alpha) + \eta}$ and $\beta_{n} \slvar n^{b}$ with $\eta, b> 0$, and any $\gamma > b$, we have w.e.h.p.,
\[
\sum_{j=1}^{n}|m_{ij}|\mathbf 1_{\{\alpha_{n} < |m_{ij}| \le \beta_{n}\}}\le n^{\gamma}.
\] 
\end{enumerate}
Moreover, we have the same bounds for the column sums, that is, the same results in \eqref{lemma3.2a}--\eqref{lemma3.2d} hold if we replace $\sum_{j=1}^{n}$ by $\sum_{i=1}^{p}$ in each part.
\end{lem}

This Lemma is modified from \cite[Proposition A.6]{BenaychGeorges:2014ki}. The main difference is that in  \cite[Proposition A.6]{BenaychGeorges:2014ki}, because each row has asymptotically $\slvar n^{\mu}$ nonzero entries (nonrandom), the summation in each part ran through only the $\slvar n^{\mu}$ nonzero terms, whereas in our setting, the number of nonzero terms in each row and column and their positions are  random and hence we include every term in a row or column. However, the proof strategy is similar and relies on the following consequence of Bennett's inequality \cite{Bennett:1962ev} (see also \cite[Lemma A.7]{BenaychGeorges:2014ki}).
\begin{lem}\label{lem:chernoff}
For each $n \ge 1$, let $X_{1}, \ldots, X_{m}$ be independent Bernoulli random variables with parameter $p$, where $m$ and $p$ depend on the parameter $n$. If $mp > Cn^{\theta}$ for some constants $C, \theta > 0$, then for any $\eta > 0$, w.e.h.p., 
\[
\left| \frac{1}{m} \sum_{i=1}^{m}X_{i} - p \right| \le \eta p.
\]
\end{lem}

\begin{proof} [Proof of Lemma \ref{lem:keylemma2}] $ $
\begin{enumerate}[(a)]
\item Assume first that $\mu >0$. Choose $\epsilon_{0} \in (0, \epsilon)$ such that $b/\epsilon_{0} \not \in \mathbb Z$. Let $T=\lfloor b/\epsilon_{0}\rfloor$, then $T\epsilon_{0} < b$. Hence we may choose 
\[
\theta = \frac{\mu - \alpha T \epsilon_{0}}{2} > \frac{\mu - \alpha b}{2} \ge \frac{\mu-\mu}{2} =0. 
\]
For each $k=0,1,\ldots, T$, define 
$
Y_{i}^{(k)} \mathdef \#\{|m_{ij}|: n^{k\epsilon_{0}} < |m_{ij}| \le n ^{(k+1)\epsilon_{0}}\}
$. 
The summation on the left side of \eqref{eq:lemHiSi} is bounded by
\begin{align*}
\sum_{j=1}^{n}|m_{ij}| \mathbf 1_{\{0 < |m_{ij}|\le 1\}} + \sum_{k=0}^{T} Y_{i}^{(k)}n^{(k+1)\epsilon_{0}} 
= \text{(I)} + \text{(II)}.
\end{align*}
Note that (I) is bounded by $\sum_{j=1}^{n}\mathbf 1_{\{0 < |m_{ij}|\}}$, and by Lemma \ref{lem:chernoff}, we have, w.e.h.p.,
\begin{equation}\label{eq:anotherone}
\text{(I)} \le \sum_{j=1}^{n}\mathbf 1_{\{|m_{ij}|>0\}} \le 2n^{\mu} = o(n^{\mu + \epsilon}).
\end{equation}
Moreover, 
$ Y_{i}^{(k)} \le  \sum_{j=1}^{n}\mathbf 1_{\{n^{k\epsilon_{0}} < |m_{ij}| \}},$
where each $\mathbf 1_{\{n^{k\epsilon_{0}} < |m_{ij}| \}}$ is an independent copy of a Ber($ L(n^{k\epsilon_{0}})n^{\mu-1 -\alpha k \epsilon_{0}}$) random variable. Since $\mu - \alpha k \epsilon_{0} > \theta > 0$ for all $k=0, \ldots, T$, we know that w.e.h.p., by Lemma  \ref{lem:chernoff},
\[
Y^{(k)}_{i} \le 2L(n^{k\epsilon_{0}})n^{\mu  -\alpha k \epsilon_{0}}.
\]
Thus, w.e.h.p., (II) is no more than
\[
L(n^{k\epsilon_{0}})\sum_{k=0}^{T} 2n^{\mu-\alpha k \epsilon_{0}} n^{(k+1)\epsilon_{0}} = 2L(n^{k\epsilon_{0}})n^{\mu+\epsilon_{0}}\sum_{k=0}^{T}n^{(1-\alpha)k\epsilon_{0}}.
\]
If $0 < \alpha \le 1$,  then $n^{(1-\alpha)k\epsilon_{0}} \le n^{(1-\alpha)T\epsilon_{0}}$. Since $b > T\epsilon_{0}$ and $\epsilon_{0} < \epsilon$, we have
\begin{align}
\notag
\text{(II)}&\le 2L(n^{k\epsilon_{0}})n^{\mu+\epsilon_{0}}(T+1)n^{(1-\alpha)T\epsilon_{0}}\\
\label{eq:HelloTuca} & \le 2L(n^{k\epsilon_{0}})(T+1) n ^{\mu + \epsilon_{0} + (1-\alpha )b} < n^{\mu+b(1-\alpha)+\epsilon}.
\end{align}
If $\alpha > 1$, then $n^{(1-\alpha)k\epsilon_{0}} \le 1$. Using again $\epsilon_{0} < \epsilon$, we obtain,
\begin{equation}\label{eq:HelloSi}
\text{(II)}\le 2L(n^{k\epsilon_{0}})n^{\mu+\epsilon_{0}} \cdot (T+1) \le n^{\mu+\epsilon}.
\end{equation}
Combining \eqref{eq:anotherone}, \eqref{eq:HelloTuca} and \eqref{eq:HelloSi}, we prove Part \eqref{lemma3.2a} for $\mu>0$. In the case $\mu=0$, then $b=0$ and the number $V_n$ of nonzero terms in the sum $\sum_{0<|m_{ij}|}|m_{ij}|$ converges in distribution to a Poisson random variable of mean $1$. It suffices then to bound the sum in \eqref{eq:lemHiSi} by $\beta_n V_n$ to obtain the desired result.
\item Just like in Part \eqref{lemma3.2a}, we choose $\epsilon_{0} < \epsilon$ such that $\frac{b-a}{\epsilon_{0}} \not \in \mathbb Z$, and take $T = [\frac{b-a}{\epsilon_{0}}]$. \\
Set $\theta = \frac{\mu- \alpha (a+ T\epsilon_{0})}{2} >  \frac{\mu- \alpha b}{2} \ge 0$. Define 
\[
Y_{i}^{(k)}\mathdef \# \{|m_{ij}|: \alpha_{n} n ^{k\epsilon_{0}} < |m_{ij}| \le \alpha_{n} n ^{(k+1) \epsilon_{0}}\}.
\]
Then, for each $k$, $Y_{i}^{(k)}$ is bounded by $\# \{|m_{ij}|: \alpha_{n} n ^{k\epsilon_{0}} < |m_{ij}|\}$, which is distributed as a Bin$(n, n^{\mu-1}L(\alpha_{n} n ^{k\epsilon_{0}})(\alpha_{n} n ^{k\epsilon_{0}})^{-\alpha} )$ random variable. Again, for each $k\le T$,
\begin{align*}
n \cdot n^{\mu-1}L(\alpha_{n} n ^{k\epsilon_{0}})(\alpha_{n} n ^{k\epsilon_{0}})^{-\alpha}
 \slvar n^{\mu-(a+k\epsilon_{0})\alpha} > n^{\theta}.
\end{align*}
From Lemma \ref{lem:chernoff}, w.e.h.p., $Y_{i}^{(k)} \le 2 n^{\mu-\alpha (a+k\epsilon_{0})+\delta}$, for arbitrarily small $\delta>0$. Then, we have w.e.h.p.,
\begin{align*}
&\sum_{j=1}^{n}|m_{ij}| \mathbf 1_{\{\alpha_{n} < |m_{ij}| \le \beta_{n}\}} \le \sum_{k=0}^{T}Y_{i}^{(k)}\alpha_{n} n^{(k+1)\epsilon_{0}}\le 2 \sum_{k=0}^{T} n^{\mu-\alpha (a+k\epsilon_{0})+\delta} n^{a+(k+1)\epsilon_{0}+\delta'}\\
&\le n^{\mu+\epsilon_{0} + \delta'' - a(\alpha-1)} \sum_{k=0}^{T}n^{ k\epsilon_{0}(1-\alpha)} \le n^{\mu+\epsilon_{0} + \delta'' - a(\alpha-1)} (T+1) \le n^{\mu-a(\alpha-1)+\epsilon}.
\end{align*}
\item For any $\delta, \delta' >0$, the left side is no larger than
$
n^{\frac{\mu}{\alpha} +\eta'+ \delta'} \sum_{j=1}^{n}\mathbf 1_{\{n^{\frac{\mu}{\alpha} - \eta -\delta} < |m_{ij}| \}}
$, where $ \sum_{j=1}^{n}\mathbf 1_{\{n^{\frac{\mu}{\alpha} - \eta -\delta} < |m_{ij}| \}}\sim $Bin$(n, n^{\mu-1} L'(n)n ^{(\frac{\mu}{\alpha} - \eta -\delta)(-\alpha)})$, for some slowly varying function $L'$. Since
\[
n\cdot n^{\mu-1} L'(n)n ^{(\frac{\mu}{\alpha} - \eta -\delta)(-\alpha)} =  L'(n)n^{\alpha(\eta + \delta)} \ge n^{\theta}, \quad\text{for } \theta =\alpha \eta /2 > 0,
\] 
then w.e.h.p., $ \sum_{j=1}^{n}\mathbf 1_{\{n^{\frac{\mu}{\alpha} - \eta -\delta}<|m_{ij}|\} } \le 2L'(n)n^{\alpha(\eta + \delta)}$. But $\delta, \delta' > 0$ can be chosen arbitrarily small, so as long as $\epsilon > \alpha \eta + \eta'$, w.e.h.p., the left side is bounded by
\[
n^{\frac{\mu}{\alpha}+ \eta' + \delta'}  \cdot 2L'(n)n^{\alpha(\eta + \delta)} =n^{\frac{\mu}{\alpha}+\alpha \eta + \eta'} \cdot (2L'(n)n^{\alpha\delta + \delta'})\le  n^{\frac{\mu}{\alpha} + \epsilon}.
\]

\item We compute this probability directly. Write $S_{i}\mathdef \sum_{j=1}^{n}|m_{ij}| \mathbf 1_{\{\alpha_{n} < |m_{ij}|\le \beta_{n}\}}$. \\
For any $\gamma > b$, and choose $\epsilon, \epsilon' >0$ sufficiently small (e.g., $\epsilon < \alpha \eta/2, \epsilon' < (\gamma -b )/2$),
\begin{align*}
\Pro(S_{i} > n^{\gamma}) &\le \Pro(\text{at least } \lfloor n^{\gamma}/\beta_{n} \rfloor \text{ terms are nonzero among the $n$ $m_{ij}$'s})\\
&\le n^{\lfloor n^{\gamma}/\beta_{n}\rfloor} \left[\Pro(|m_{ij}| > \alpha_{n}) \right]^{\lfloor n^{\gamma}/\beta_{n}\rfloor}\\
&= [n\cdot n^{\mu-1} L(\alpha_{n}) \alpha_{n}^{-\alpha}]^{\lfloor n^{\gamma}/\beta_{n}\rfloor }\le [n^{\mu}n^{(\mu/\alpha + \eta) (-\alpha)+\epsilon}]^{\lfloor n^{\gamma}/\beta_{n}\rfloor }\\
&\le[n^{ - \alpha\eta +\epsilon}]^{n^{\gamma-b -\epsilon'}} \le n^{(-\alpha \eta/2) \cdot n^{(\gamma-b)/2} } = e^{-n^{(\gamma-b)/2}\cdot (\alpha \eta \log n)/2} \le e^{-n^{\theta}}.
\end{align*}
\end{enumerate}
\end{proof}
\subsection{An upper bound on the trace}

We now prove an upper bound for the norm of the truncated sample covariance matrix, as given below. 

\begin{theorem}
\label{thm:traceboundcov}  Suppose $\alpha > 2$ and $x$ has mean zero and variance one. 
Let $M_{p,n}=A_{p,n}\cdot B_{p,n}$ be the sparse $p \times n$ matrix with heavy tailed entries. Consider $\gamma, \gamma'>0$ such that $\gamma' >  \gamma$ and $\gamma' \ge \frac{\mu}{2}$. Define the truncated matrix 
\[
\hat M_{p,n}  = [\hat m_{ij}]_{i, j=1}^{p,n}= [m_{ij}\mathbf 1_{\{|m_{ij}| \le n^{\gamma}\}}]_{i,j=1}^{p,n}.
\]
We also assume that the truncated entries are centered. Then for any $\kappa > 1$, 
\begin{equation*}
\Pro\bigg(||\hat M_{p,n}\hat M_{p,n}^{*}|| \ge \kappa  n^{2\gamma'} (1+\sqrt{\rho})^{2}\bigg) \to 0.
\end{equation*}
\end{theorem}
\begin{proof} 
For $\kappa > 1$ given, we find $C \in (1, \kappa)$ and let $E_{n}$ be the event 
\[
E_{n} = \{\mathsf L\le Cn^{\mu}, \mathsf {\tilde L} \le Cpn^{\mu-1}\},
\] 
where $\mathsf L$ (resp. $\mathsf {\tilde L}$) is the maximum number of nonzero entries in a row (resp. a column) among the $p$ rows (resp. $n$ columns) of $M_{p,n}$. We break the desired probability into two parts:
\begin{align*}
\Pro\bigg(||\hat M_{p,n}\hat M_{p,n}^{*}|| \ge \kappa  n^{2\gamma'} (1+\sqrt{\rho})^{2}, E_{n}\bigg) + \Pro\bigg(||\hat M_{p,n}\hat M_{p,n}^{*}|| \ge \kappa  n^{2\gamma'} (1+\sqrt{\rho})^{2}, E_{n}^{c}\bigg). 
\end{align*}
Since the second term is bounded by $\Pro(E_{n}^{c})$ which vanishes as $n \to \infty$ by Chernoff inequality, it suffices to prove that the first term vanishes as well. To do this, 
we choose $\gamma'' > 0$ such that $\gamma' > \gamma + 6\gamma''$ and set $k=k_{n}= \lfloor n^{\gamma''}\rfloor$. 
We will  prove that for any $\delta>0$ small, 
\begin{align}
\label{eqn:traceboundcov:toprove}
\E (\Tr(\hat M_{p,n} \hat M_{p,n}^{*})^{k} \mathbf 1_{E_{n}} ) \le P(n)
 \left[ (1+\sqrt{\delta})^{2}Cn^{2\gamma'}(1+\sqrt{\rho})^{2}\right]^{k},\end{align}
 where $P(n)$ is some polynomial of $n$. It will then follow from \eqref{eqn:traceboundcov:toprove} that
\begin{equation*}
\Pro\bigg(||\hat M_{p,n}\hat M_{p,n}^{*}|| \ge \kappa  n^{2\gamma'} (1+\sqrt{\rho})^{2}, E_{n}\bigg)  \le\frac{\E (\Tr(\hat M_{p,n} \hat M_{p,n}^{*})^{k } \mathbf 1_{E_{n}} )}{\kappa^{k} [ n^{2\gamma'}]^{k} (1+\sqrt{\rho})^{2k} }\le  P(n) \left[\frac{C(1+\sqrt{\delta})^{2}}{\kappa}\right]^{k}. 
\end{equation*}
The right side goes to zero as $n\to \infty$, if we choose $\delta > 0$ such that $C(1+\sqrt{\delta})^{2} < \kappa$. 
To show \eqref{eqn:traceboundcov:toprove}, we will make use of the combinatorics that was invented in \cite{Yin:1988kp} to prove the convergence of the largest eigenvalue of random sample covariance matrices. \\
We first expand the left side of~\eqref{eqn:traceboundcov:toprove}:
\begin{align*}
\E (\Tr(\hat M_{p,n} \hat M_{p,n}^{*})^{k} \mathbf 1_{E_{n}} ) = \sum_{i_{1}, i_{3},\ldots, i_{2k-1}=1}^{p}\sum_{i_{2}, i_{4},\ldots, i_{2k}=1}^{n}\E (\hat m_{i_{1}i_{2}}\hat m_{i_{3}i_{2}}\cdots \hat m_{i_{2k-1}i_{2k}}\hat m_{i_{1}i_{2k}} \mathbf 1_{E_{n}}).
\end{align*}
Then, we associate each summand on the right side with an undirected graph $G_{\mathbf i}$ that has vertices $\{i_{1}, \ldots, i_{2k}\}$ and edges $\{(i_{1},i_{2}), (i_{2},i_{3}), \ldots, (i_{2k-1},i_{2k}), (i_{2k},i_{1}) \}$. We read the vertices  sequentially, $i_{1},i_{2},i_{3},\ldots, i_{2k}$, one at a time, and classify the edges into four different types. We call an edge $(i_{s-1}, i_{s}), s \ge 2$,  an {\em innovation} if $i_{s}$ does not occur in $i_{1}, \ldots, i_{s-1}$. An innovation $(i_{s-1},i_{s})$ is called a {\em row innovation} if $s$ is odd and  a {\em column innovation} if $s$ is even. If two edges $(i_{a-1}, i_{a})$ and $(i_{b-1}, i_{b})$ have the same set of vertices, i.e., $\{i_{a-1}, i_{a}\} = \{i_{b-1}, i_{b}\}$, we say that they {\em coincide}.  And an edge $(i_{a-1}, i_{a})$ is said to be {\em single up to $i_{b}$}, with $b \ge a$,  if there is no other edge $(i_{c-1}, i_{c})$ with $2\le c \le b$ that coincides with $(i_{a-1}, i_{a})$. For $b \ge 3$, we call $(i_{b-1}, i_{b}) $ a {\em $T_{3}$-edge} if there is an innovation $(i_{a-1}, i_{a}), a <b,$ that is single up to $i_{b-1}$ and coincides with $(i_{b-1}, i_{b})$. And finally, an edge is called a {\em $T_{4}$-edge} if it is neither a $T_{3}$-edge nor an innovation. Hence, observing $\hat m_{ij} = a_{ij}\mathbf 1_{\{|a_{ij}| \le n^{\gamma}\}} \cdot b_{ij} \mathdef \hat a_{ij} b_{ij}$ and using independence, the expectation can be rewritten as
\begin{align*}
\E (\Tr(\hat M_{p,n} \hat M_{p,n}^{{*}})^{k} \mathbf 1_{E_{n}} ) &= {\textstyle \large \sum' \sum'' \sum'''} \E (\hat m_{i_{1}i_{2}}\hat m_{i_{3}i_{2}}\cdots \hat m_{i_{2k-1}i_{2k}}\hat m_{i_{1}i_{2k}} \mathbf 1_{E_{n}})\\
&= {\textstyle \large \sum' \sum'' \sum'''} \E (\hat a_{i_{1}i_{2}}\hat a_{i_{3}i_{2}}\cdots \hat a_{i_{1}i_{2k}} )\E(b_{i_{1}i_{2}} b_{i_{3}i_{2}}\cdots b_{i_{1}i_{2k}}\mathbf 1_{E_{n}})
\end{align*}
where $\sum'$ sums over all possible arrangements of the four types of the edges, $\sum''$ is to count the total number of different canonical graphs given the arrangements of the four types of edges, and $\sum'''$ runs through all graphs that are isomorphic to the given canonical graph. 

Let $l$ be the number of $T_{3}$-edges. Note $l$ is also the number of innovations since every edge must be visited at least twice, and hence $(2k-2l)$ is the number of $T_{4}$-edges. Let $r$ be the number of row innovations. We see that $\sum'$ is bounded by $\sum_{l=1}^{k}\sum_{r=0}^{l}\binom{k}{r}\binom{k}{l-r}\binom{2k-l}{l}$. Since every row innovation $(i_{2s-2}, i_{2s-1})$ leads to a new vertex $i_{2s-1} \in \{1,2,\ldots, p\}$ and every column innovation $(i_{2s-1}, i_{2s})$ leads to a new vertex $i_{2s} \in \{1,2, \ldots, n\}$, except for the first innovation $(i_{1}, i_{2})$, which leads to both a new vertex $i_{1} \in \{1,2, \ldots p\}$ and a new vertex $i_{2} \in \{1,2,\ldots, n\}$, then, on the event $E_{n}$, there are at most $p(Cpn^{\mu-1})^{r}(Cn^{\mu})^{l-r}$ terms that have nonzero contributions to $\sum'''$. Let $q$ be the number of distinct $T_{4}$-edges. It was shown in \cite[page 519]{Yin:1988kp} that $\sum''$ is bounded by $k^{2q}(q+1)^{6k-6l}$.

Finally, let $b$ be the number of $T_{4}$-edges among the $q$ distinct ones that coincide with some innovations, and let $n_{s}, s=1,2,\ldots, b$, be the multiplicity of  the $T_{4}$-edges of the $s$-th such coincidence. Then, $(q-b)$ distinct $T_{4}$-edges do not coincide with any innovations but only among the $T_{4}$-edges, and we denote by $m_{t}, t=1,2,\ldots, (q-b)$ the multiplicity of the $t$-th such coincidence. These numbers have to satisfy the relation
$
2k~-~2l~=~\sum_{s=1}^{b}n_{s}~+~\sum_{t=1}^{q-b}m_{t}
$. 
Hence, for such composition of four types of edges, we can write 
\begin{align*}
&\quad \E (\hat a_{i_{1}i_{2}}\hat a_{i_{3}i_{2}}\cdots \hat a_{i_{2k-1}i_{2k}}\hat a_{i_{1}i_{2k}}) \\
&= 
\left(\E \hat a_{11}^{2} \right)^{l-b} \prod_{s=1}^{b}\left( \E \hat a_{11}^{n_{s}+2}\right)\prod_{t=1}^{q-b}\left( \E \hat a_{11}^{m_{t}}\right) \\
&\le L_{0}(n^{\gamma})^{q} \prod_{\substack{s=1\\ n_{s}+2 > \alpha}}^{b}\left(\frac{n_{s}+2}{n_{s}+2-\alpha}\right) n^{\gamma (n_{s}+2-\alpha)}\prod_{\substack{t=1\\ m_{t} > \alpha}}^{q-b}\left(\frac{m_{t}}{m_{t}-\alpha}\right) n^{\gamma (m_{t}-\alpha)},
\end{align*}
where $L_{0}$ is a slowly varying function. Here, in the last inequality, we have used the following classic fact for the moments of truncated, heavy tailed random variables (see, e.g., \cite[Proposition 1.5.8]{Regular} or \cite[Lemma A.8]{BenaychGeorges:2014ki}),
\[
\E |a_{11}|^{s}\mathbf 1_{|a_{11}|\le x} = \left\{
\begin{array}{ll}
L_{0}(x), & \text{ if } s\le \alpha\\
L_{0}(x) \frac{s}{s-\alpha} x^{s-\alpha}, & \text{ if } s > \alpha
\end{array}
\right..
\]
Observing that (i) $n_{s}\le 2k-2$, $2\le m_{t} \le 2k-2$, (ii) $\frac{1}{m-\alpha}$ is bounded above, say, by some $C_{\alpha}$, which only depends on $\alpha$, for all $m > \alpha$, and  $m \in \mathbb Z$, and  (iii) $\alpha > 2$, we have 
\begin{align*}
\E (\hat a_{i_{1}i_{2}}\hat a_{i_{3}i_{2}}\cdots \hat a_{i_{2k-1}i_{2k}}\hat a_{i_{1}i_{2k}}) 
&\le L_{0}(n^{\gamma})^{q} (2C_{\alpha} k)^{q} n^{\gamma [\sum_{s=1}^{b}(n_{s}+2-\alpha)^{+} + \sum_{t=1}^{q-b}(m_{t}-\alpha)^{+}]}\\
&\le L_{0}(n^{\gamma})^{q} (2C_{\alpha} k)^{q} n^{\gamma [\sum_{s=1}^{b}(n_{s}+2-2)^{+} + \sum_{t=1}^{q-b}(m_{t}-2)^{+}]}\\
&\le L_{0}(n^{\gamma})^{q} (2C_{\alpha} k)^{q} n^{\gamma [(2k-2l) - 2(q-b)]}
\end{align*}
where $f^{+} = \max (f, 0)$. After reorganizing and combining the terms, the expectation of the trace is then bounded by
\begin{align*}
\E (\Tr(\hat M_{p,n} \hat M_{p,n}^{*})^{k} \mathbf 1_{E_{n}} ) \le & p\sum_{l=1}^{k}\sum_{r=0}^{l}\binom{k}{r}\binom{k}{l-r}\binom{2k-l}{l}(Cpn^{\mu-1})^{r}(Cn^{\mu})^{l-r} \\
&\cdot \left(\sum_{q=0}^{2k-2l} 
(q+1)^{6k-6l} (2L_{0}(n^{\gamma})C_{\alpha}k^{3})^{q}n^{\gamma(2k-2l-2q)} \sum_{b=0}^{q}  n^{2\gamma b} \right).
\end{align*}
We now consider the terms inside the bracket. Firstly,  for $\gamma > 0$,
\begin{align*}
\sum_{b=0}^{q}  n^{2\gamma b} = \frac{n^{2\gamma(q+1)}-1}{n^{2\gamma}-1} \le  \frac{n^{2\gamma(q+1)}}{n^{2\gamma}/2} =2n^{2\gamma q}.
\end{align*}
Next, we use the elementary inequality $(q+1)^{z}\le w^{q+1}(z/\log w)^{z}$ for any $w > 1, z>0, q>0$. In what follows, we apply this inequality, substituting $w=2$ and $z=6k-6l$, and $L_{j}(\cdot)$'s are all slowly varying functions.  
\begin{align*}
& \sum_{q=0}^{2k-2l} (q+1)^{6k-6l} (2L_{0}(n^{\gamma})C_{\alpha}k^{3})^{q}n^{\gamma(2k-2l-2q)} \sum_{b=0}^{q}  n^{2\gamma b}\\ 
\le &\  \sum_{q=0}^{2k-2l} 2^{q+1}(6k/\log 2)^{6k-6l} (L_{1}(n^{\gamma})k^{3})^{q}n^{\gamma(2k-2l-2q)} n^{2\gamma q}\\
\le &\  2(6n^{\gamma/3}k/\log 2)^{6k-6l} \cdot \sum_{q=0}^{2k-2l} (2L_{1}(n^{\gamma})k^{3})^{q}\\
\le & \ (6n^{\gamma/3}k/\log 2)^{6k-6l} (L_{2}(n^{\gamma})k^{3})^{2k-2l}\le \left(k^{2} (n^{\gamma}L_{3}(n^{\gamma}))^{1/3} \right)^{6k-6l}. 
\end{align*}
Next, using the combinatorial inequality that for any $\delta > 0$ (see \cite[Lemma 2.1]{Yin:1988kp}),
\[
\binom{k}{r}\binom{k}{l-r}\binom{2k-l}{l} \le (1+\sqrt{\delta})^{2k}\delta^{l-k}\binom{k}{l}\binom{2l}{2r},
\]
we get
\begin{align*}
& \E (\Tr(\hat M_{p,n} \hat M_{p,n}^{*})^{k} \mathbf 1_{E_{n}} ) \\
 &\ \le p (1+\sqrt{\delta})^{2k} \sum_{l=1}^{k} \binom{k}{l} \sum_{r=0}^{l} \binom{2l}{2r} (p/n)^{r}(Cn^{\mu})^{l} \left(k^{2} (n^{\gamma}L_{3}(n^{\gamma}))^{1/3} \delta^{-1/6}\right)^{6k-6l} \\
 &\ \le P(n) (1+\sqrt{\delta})^{2k} \sum_{l=1}^{k} \binom{k}{l}(1+\sqrt{p/n})^{2l}(Cn^{\mu})^{l} \left(k^{2} (n^{\gamma}L_{3}(n^{\gamma}))^{1/3} \delta^{-1/6}\right)^{6k-6l} \\
 &\ \le P(n)  \left[ (1+\sqrt{\delta})^{2}Cn^{\mu}(1+\sqrt{p/n})^{2} + \left(k^{2} (n^{\gamma}L_{4}(n^{\gamma}))^{1/3} \delta^{-1/6} \right)^{6} \right]^{k}.
\end{align*}
Since $\mu \le 2\gamma' $, and for any $\delta >0$, $k = \lfloor n^{\gamma''} \rfloor $, 
\[
\left(k^{2} (n^{\gamma}L_{4}(n^{\gamma}))^{1/3} \delta^{-1/6} \right)^{6} \slvar n^{2(\gamma +6\gamma'')} = o(n^{2\gamma'}),
\] 
we get \eqref{eqn:traceboundcov:toprove}. 
\end{proof}

\begin{rem}\label{bandremark}
All the proofs in this subsection (and further on) only used sparseness to determine the number of nonzero entries in each row and column. The exact location of the nonzero entries plays no role in the proof. Hence, all our results hold for a larger class of sample covariance matrices including those constructed from banded rectangular matrices.  
\end{rem}

\subsection{Perturbations of eigenvalues and eigenvectors.} The next two lemmas are classical tools of perturbation theory of eigenvalues. 
\begin{theorem}[Cauchy interlacing theorem] \label{thm:cauchy} Let $1\le p\le n$. 
\begin{enumerate}[\normalfont (a)]
\item \label{lemma3.5a}Let $A_{n}$ be an $n \times n$ Hermitian matrix and $A_{n-1}$ be its $(n-1)\times (n-1)$ minor, then $\lambda_{1}(A_{n}) \ge \lambda_{1}(A_{n-1}) \ge \lambda_{2}(A_{n}) \ge \cdots \ge \lambda_{n-1}(A_{n-1}) \ge \lambda_{n}(A_{n})$;
\item \label{lemma3.5b} Let $A_{p, n}$ be a $p \times n$ matrix and $A_{(p-1), n}$ be its $(p-1)\times n$ minor, then\\ $\sigma_{1}(A_{p, n}) \ge \sigma_{1}(A_{(p-1), n}) \ge \sigma_{2}(A_{p, n}) \ge \cdots \ge \sigma_{p-1}(A_{(p-1), n}) \ge \sigma_{p}(A_{p, n})$;
\item \label{lemma3.5c} If $p<n$ and  $A_{p, n}$ is a $p \times n$ matrix and $A_{p, (n-1)}$ is its $p\times (n-1)$ minor, then\\ $\sigma_{1}(A_{p, n}) \ge \sigma_{1}(A_{p, (n-1)}) \ge \sigma_{2}(A_{p, n}) \ge \cdots \ge \sigma_{p-1}(A_{p, (n-1)}) \ge \sigma_{p}(A_{p, n})$,
\end{enumerate}
where in (b) and (c), $\sigma_{i}(\cdot)$ denotes the $i$-th largest singular value.
\end{theorem}
\begin{proof} See for instance \cite [Lemma 22]{Tao:2012ix}.
\end{proof}
\begin{theorem}[Perturbation of eigenvalues and eigenvectors] \label{thm:pertubeigs} Let $A$ be a Hermitian matrix and $\mathbf v$ be a unit vector. Let $\zeta =\langle \mathbf v, A\mathbf v\rangle$ and $\epsilon = ||(A-\zeta)\mathbf v||$.
\begin{enumerate}[\normalfont (a)]
\item There exists an eigenvalue $\lambda_{\epsilon}$ of $A$ in the closed ball $\overline{ B(\zeta, \epsilon)}$. 
\item If $\lambda_{\epsilon}$ is the only eigenvalue in $\overline{B(\zeta, \epsilon)}$ with corresponding eigenvector $\mathbf v_{\epsilon}$, and all other eigenvalues are at distance at least $d > \epsilon$ of $\zeta$ then $ \Vert  \mathbf v _{\epsilon} - P_{\mathbf v}(\mathbf v_{\epsilon})\Vert \le \frac{2\epsilon}{d-\epsilon}$. 
\end{enumerate}
\end{theorem}
\begin{proof} A proof can be found in \cite[page 77]{Bhatia:1997tf} or \cite[Proposition A.1]{BenaychGeorges:2014ki}.
\end{proof}
\subsection{Convergence of ESD} In this subsection, we state the convergence of the corresponding empirical spectral measures. We assume $\mu>0$ for the next proposition.
\begin{prop}\label{prop:samplecovariance} Suppose $\alpha > 2$, $\mu \in (0,1]$ and $x$ has variance one. Let $\Sigma_{p,n}$ be the sparse heavy tailed sample covariance matrix. Then the empirical spectral distribution of $\Sigma_{p,n}/n^{\mu}$ converges almost surely to the Marchenko-Pastur law with density
\[
 \frac{\sqrt{(\lambda_+-x)(x-\lambda_-)}}{2 \pi\rho x }\mathbf 1_{[\lambda_-,\lambda_+]}(x)
\]
with $\lambda_{\pm} = (1\pm \sqrt{\rho})^2$.
\end{prop}
\begin{proof}[Proof of Proposition \ref{prop:samplecovariance}] The proof  follows from the classic truncation and moment method for random matrices (See for instance \cite[Exercise 2.1.18]{Anderson:2010up}). Normalizing the entries $m_{ij}$ by $n^{\mu/2}$ gives the desired variance:
\[
\Var (m_{ij}/n^{\mu/2}) = n^{-\mu}\Var(m_{ij}) = n^{-\mu} \cdot n^{\mu-1} = \frac{1}{n}.
\]
\end{proof}

\section{Proof of the main theorems}\label{sec4}
\subsection{Proof of Theorem \ref{thm:mainresult3} } 
\subsubsection{Proof strategy} 

We will use the strategy that was first proposed by Soshnikov  \cite{Soshnikov:2004uc} when proving the heavy tailed Hermitian matrix case with $0<\alpha < 2$. This idea was later developed in \cite{Auffinger:2009vs} for proving the Hermitian case and the sample covariance matrix case when $0 < \alpha < 4$ and used in \cite{BenaychGeorges:2014ki} for proving the band Hermitian matrix case when  $\alpha > 0$. 

The strategy is as follows. We first show that the convergence holds when $l =1$, i.e., $\frac{\lambda_{1}(\Sigma_{p,n})}{|m_{i_{1}j_{1}}|^2}\stackrel{P}{\rightarrow} 1$. Then, we remove the $i_{1}$-th row from $M_{p,n}$. Lemma \ref{lem:keylemma1} guarantees that, with high probability, the second largest entry will not be removed. The convergence for the second largest eigenvalue and the second largest entry follows from Theorem~\ref{thm:cauchy} and the same argument for the $l=1$ case. Iterating this process, one proves $\frac{\lambda_{l}(\Sigma_{p,n})}{|m_{i_{l}j_{l}}|^2}\stackrel{P}{\rightarrow} 1$ for each $l$ fixed.

\subsubsection{Eigenvalues}
We begin by computing the two-sided tail probability of $|m_{ij}|^{2}$. For any $t >0$,
\begin{align*}
\Pro(|m_{ij}|^{2} \ge  t) &= \Pro(|m_{ij}| > \sqrt{t}) = L(\sqrt{t})n^{\mu-1} (\sqrt{t})^{-\alpha}=n^{\mu-1}L(\sqrt{t})t^{-\alpha/2}.
\end{align*}
Since $L(\sqrt{t})$ is also a slowly varying function in $t$,  $[|m_{ij}|^{2}]_{i,j=1}^{p,n}$ is a sparse heavy tailed random matrix of $p\times n$ independent entries with parameter $\mu$ and $\alpha/2$. Classic extreme value theory tells us that the random point process $\mathcal Q_{n}$, defined in (\ref{eqn:randppcov1}), converges to the desired Poisson point process with intensity $\alpha/2x^{1+\alpha/2}$. In particular, $c_{np}^{-2}|m_{i_{1}j_{1}}|^{2}$ converges to a Frech\'et distribution with parameter $\alpha/2$. \Br
We next show that the largest eigenvalue of $\Sigma_{p,n}$ behaves like the square of the largest entry of $M_{p,n}$ ($l=1$ case), i.e., $\frac{\lambda_{1}(\Sigma_{p,n})}{|m_{i_{1}j_{1}}|^{2}} \stackrel{P}{\longrightarrow} 1$. 
Since $\Sigma_{p,n}$ is positive semidefinite, $\lambda_{1}(\Sigma_{p,n}) \ge \langle \Sigma_{p,n} \mathbf v, \mathbf v\rangle = \mathbf v^{*}M_{p,n}M_{p,n}^{*}\mathbf v$ for any unit vector $\mathbf v$.  Hence, we can choose $\mathbf v=\mathbf e_{i_{1}} $, which gives
\[
\lambda_{1}(\Sigma_{p,n}) \ge  \mathbf v^{t}M_{p,n}M_{p,n}^{*}\mathbf v = \sum_{j=1}^{n}|m_{i_{1}j}|^{2} = |m_{i_{1}j_{1}}|^{2} + \sum_{j\ne j_{1},j=1}^{n}|m_{i_{1}j_{1}}|^{2}\ge |m_{i_{1}j_{1}}|^{2}, 
\]
and it suffices to prove the reverse direction, i.e., $\forall \epsilon > 0$, 
\begin{equation}
\label{eqn:sampcovub}
\Pro(\lambda_{1}(\Sigma_{p,n})> |m_{i_{1}j_{1}}|^{2} (1+\epsilon)) \to 0,\quad \text{as } n \to \infty.
\end{equation}
We use the infinity norm of $\Sigma_{p,n}$ to bound $\lambda_{1}(\Sigma_{p,n})$ and truncate the matrix $M_{p,n}$, when necessary.  \br
\textit{$\bullet$ Case I: {$0 < \alpha < 1+\mu^{-1}$}}.\\
In this case, we can directly show (\ref{eqn:sampcovub}). Observing that
\begin{align*}
\lambda_{1}(\Sigma_{p,n})  = ||\Sigma_{p,n}|| = ||M_{p,n}M_{p,n}^{*}|| \le ||M_{p,n}||^{2} \le  ||M_{p,n}||_{\infty} ||M_{p,n}||_{1} 
\end{align*}
it suffices to show that with probability tending to one,
\begin{align} 
\label{eq:double1}
||M_{p,n}||_{\infty} \le |m_{i_{1}j_{1}}| (1+o(1)),\\
\label{eq:double2}
||M_{p,n}||_{1} \le |m_{i_{1}j_{1}}| (1+o(1)).
\end{align}
The proof of \eqref{eq:double2} will be almost identical to \eqref{eq:double1}, by switching the role of $p$ and $n$. We hence show \eqref{eq:double1} only.  
Lemma \ref{lem:keylemma1} \eqref{lemma3.1a} says, with probability going to 1, there is no row that has two entries with absolute value greater than $c_{np}^{\kappa}$, where $\kappa = \frac{1+2\mu}{2+2\mu} + \delta$, and $\delta > 0$ can be chosen arbitrarily small. Consider the following summation and break it into three pieces, 
\begin{align*}
&\tilde S_{i} \mathdef \sum_{j=1}^{n}|m_{ij}|\mathbf 1_{\{|m_{ij}| \le c_{np}^{\kappa}\}}\\
& = \sum_{j=1}^{n}|m_{ij}|\mathbf 1_{\{|m_{ij}| \le n^{\frac{\mu}{\alpha} -\eta}\}} +  \sum_{j=1}^{n}|m_{ij}|\mathbf 1_{\{n^{\frac{\mu}{\alpha} - \eta} < |m_{ij}| \le  n^{\frac{\mu}{\alpha}  + \eta}\}} +  \sum_{j=1}^{n}|m_{ij}|\mathbf 1_{\{ n^{\frac{\mu}{\alpha}+\eta} < |m_{ij}| \le c_{np}^{\kappa}\}}\\
&=\tilde S_{i,1} + \tilde S_{i, 2} + \tilde S_{i, 3},
\end{align*}
where we choose $\eta \in(0, \min\{\frac{1}{2\alpha(\alpha +1)}, \frac{\mu}{\alpha}\})$.\\ 
By Lemma \ref{lem:keylemma2} \eqref{lemma3.2a}, w.e.h.p., for any $\epsilon > 0$,
$
\tilde S_{i,1} \le n^{\mu + (\frac{\mu}{\alpha}-\eta)(1-\alpha)^{+} + \epsilon}$,
which is $o(n^{\frac{\mu+1}{\alpha}} )$. To see this, if $\alpha < 1$, set $\epsilon =\eta(1-\alpha) > 0$, and w.e.h.p.,
\[
\tilde S_{i,1} \le n^{\mu + (\frac{\mu}{\alpha}-\eta)(1-\alpha)+ \epsilon} = n^{\frac{\mu}{\alpha}} = o(n^{\frac{\mu+1}{\alpha}} ). 
\]
If $1\le \alpha < 1+\mu^{-1}$, then $(1-\alpha)^{+} = 0$ and hence
$
\tilde S_{i,1} \le n^{\mu + \epsilon} 
$. 
But $\alpha < 1+\mu^{-1}$ implies $ \mu < \frac{\mu+1}{\alpha}$, so for $\epsilon$ sufficiently small, $\tilde S_{i, 1} \le n^{\mu+\epsilon} = o(n^{\frac{\mu+1}{\alpha}} )$. \\
By Lemma \ref{lem:keylemma2} \eqref{lemma3.2c}, w.e.h.p.,
$
\tilde S_{i,2} \le n^{\mu/\alpha + \epsilon}$, $ \forall \epsilon > \eta(\alpha+1)$.  Since $\eta < \frac{1}{2\alpha(\alpha+1)}$, we can choose $\epsilon = \frac{1}{2\alpha}$, and this gives w.e.h.p., $\tilde S_{i, 2} \le n^{\frac{\mu}{\alpha}+\frac{1}{2\alpha}}=o(n^{\frac{\mu+1}{\alpha}})$, as desired. \\
Finally, since $\frac{1}{2}+\delta \le \kappa \le \frac{3}{4}+\delta<1, c_{np}\slvar n^{\frac{\mu+1}{\alpha}}$, and by Lemma \ref{lem:keylemma2} \eqref{lemma3.2c}, we can choose $\frac{(\mu+1)\kappa}{\alpha}< \gamma < \frac{\mu+1}{\alpha}$ such that $\tilde S_{i,3} \le n^{\gamma} = o(n^{\frac{\mu+1}{\alpha}})$. Hence,  w.e.h.p., $\tilde S_{i} = o(n^{\frac{\mu+1}{\alpha}})$. \\
The sum of absolute values in row $i$ of $M_{p,n}$ can be written as
\[
S_{i} \mathdef \sum_{j=1}^{n}|m_{ij}| = \tilde S_{i} + \sum_{j=1}^{n}|m_{ij}|\mathbf 1_{\{|m_{ij}| \ge c_{np}^{\kappa}\}}.
\]
Moreover, a crude union bound and Lemma \ref{lem:keylemma1}\eqref{lemma3.1b} give us 
\begin{align*}
\Pro \left(\exists i, \max_{1\le j \le n}|m_{ij}| \ge c_{np}^{\kappa}, S_{i}-\max_{1\le j \le n}|m_{ij}|  > c_{np}^{1-\epsilon_{0}}\right) \to 0
\end{align*}
for some $\epsilon_{0}>0$ sufficiently small, which implies \eqref{eq:double1}. Hence we have proved that $$\frac{\lambda_{1}(\Sigma_{p,n})}{|m_{i_{1}j_{1}}|^{2}} \stackrel{P}{\longrightarrow}1.$$
Next, we show that, with probability tending to one, $\Sigma_{p,n}$ has eigenvalues at $|m_{i_{l}j_{l}}|^{2}(1+o(1))$. We compute the $l$-th residual $\mathbf r_{l}$, for $l \ge 1$, i.e., 
\begin{align}\label{Sunnyoutsidebutpreferhere}
\Sigma_{p,n} \mathbf e_{i_{l}} = |m_{i_{l}j_{l}}|^{2} \mathbf e_{i_{l}} + \mathbf r_{l},
\end{align}
and hence
\[
\mathbf r_{l} = \left(\sum_{k=1}^{n}m_{1k} m_{i_{l}k},\ \ldots\ , \sum_{\substack{k=1\\k\ne j_{l}}}^{n}| m_{i_{l}k}|^{2}, \ldots,  \sum_{k=1}^{n}m_{pk} m_{i_{l}k}\right)^{T}.
\]
The norm of $\mathbf r_{l}$ is $o(c_{np}^{2})$. To see this, we compute the $||\mathbf r_{l}||$ explicitly:
\begin{align*}
||\mathbf r_{l}||_{2}&=\left(\sum_{\substack{s=1\\s\ne i_{l}}}^{p} \left[\sum_{k=1}^{n}m_{sk}m_{i_{l}k}\right]^{2} + \left[\sum_{\substack{k=1\\k\ne j_{l}}}^{n}| m_{i_{l}k}|^{2} \right]^{2} \right)^{1/2} \le  \sum_{k=1}^{n} |m_{i_{l}k}|\sum_{\substack{s=1\\s\ne i_{l}}}^{p} |m_{sk}|   + \sum_{\substack{k=1\\k\ne j_{l}}}^{n}| m_{i_{l}k}|^{2}\\
&\le \sum_{\substack{k=1\\k\ne j_{l}}}^{n} |m_{i_{l}k}|\sum_{\substack{s=1\\s\ne i_{l}}}^{p} |m_{sk}| + |m_{i_{l}j_{l}}| \sum_{\substack{s=1\\s\ne i_{l}}}^{p} |m_{sj_{l}}| + \left(\sum_{\substack{k=1\\k\ne j_{l}}}^{n}| m_{i_{l}k}|\right)^{2}.
\end{align*}
By Lemma \ref{lem:keylemma1} and \ref{lem:keylemma2} \eqref{lemma3.2a},  \eqref{lemma3.2c} and  \eqref{lemma3.2d}, and a calculation similar to that when we bound the row sum of $M_{p,n}$, one can see that $||\mathbf r_{l}||=o(c_{np}^{2})$, with probability tending to one. Hence, letting $\mathbf r_{l}' = c_{np}^{-2}\mathbf r_{l}$, we have
\[
c_{np}^{-2}\Sigma_{p,n} \mathbf e_{i_{l}} = c_{np}^{-2}|m_{i_{l}j_{l}}|^{2}\mathbf e_{i_{l}} + \mathbf r'_{l},
\]
where $||\mathbf r'_{l}|| \to 0$. It then follows from Lemma \ref{thm:pertubeigs} that $\Sigma_{p,n}$ has eigenvalues $|m_{i_{l}j_{l}}|^{2}(1+o(1))$.  Hence, with probability tending to one, $\lambda_{l}(\Sigma_{p,n}) \ge |m_{i_{l}j_{l}}|^{2}(1+o(1))$. \\
To show that these are exactly the largest eigenvalues (where the case $l=1$ is proved), we use Theorem \ref{thm:cauchy}. When $l=2$, let $M_{p,n, -i_{1}}$ be the submatrix of $M_{p,n}$ removing the $i_{1}$-th row and let $\Sigma_{p,n}^{(i_{1})} \mathdef M_{p,n, -i_{1}}M_{p,n, -i_{1}}^{*}$. 
By Lemma \ref{lem:keylemma1}\eqref{lemma3.1a}, with probability going to one, the second largest entry of $M_{p,n}$ (in absolute value), $m_{i_{2}j_{2}}$, will remain in $M_{p,n, -i_{1}}$. Using the infinity norm bound on $\Sigma_{p,n}^{(i_{1})}$ and the same argument as we prove for $\lambda_{1}(\Sigma_{p,n})$, we have, with probability tending to one, 
\[
\lambda_{2}(\Sigma_{p,n}) \le  \lambda_{1}(\Sigma_{p,n}^{(i_{1})})  = |m_{i_{2}j_{2}}|^{2}(1+o(1)), 
\]
where the first inequality is due to the interlacing of eigenvalues.
The claim for general $\lambda_{l}(\Sigma_{p,n})$ then follows from iterating the above argument.\Br
\textit{$\bullet$ Case II: {$1+\mu^{-1} \le  \alpha < 2(1+\mu^{-1})$}}.\\
 For this case,  in order to show \eqref{eqn:sampcovub}, we choose $\gamma, \gamma' > 0$ such that 
 \[
0\le  \frac{\mu}{\alpha}-\frac{1}{\alpha(\alpha-1)} < \gamma < \frac{\mu+1}{\alpha}, \quad \max\left(\gamma, \frac{\mu}{2}\right) < \gamma' < \frac{\mu+1}{\alpha},
 \]
 which is always possible if $1+\mu^{-1} \le \alpha < 2(1+\mu^{-1})$.  We truncate the entries of  $M_{p,n}$ at $n^{\gamma}$. Let $\hat M_{p,n} \mathdef [\hat m_{ij}]_{i,j=1}^{p,n} = [m_{ij}\mathbf 1_{\{|m_{ij}| \le n^{\gamma}\}}]_{i,j=1}^{p,n}$, and $M_{p,n}' = M_{p,n} - \hat M_{p,n}$ be the truncated part and the remaining part of $M_{p,n}$, respectively. 

We decompose $\Sigma_{p,n}$ as below:
\begin{align*}
\Sigma_{p,n} &= (\hat M_{p,n} + M_{p,n}')(\hat M_{p,n}+M_{p,n}')^{*} \\
 \notag &= (\hat M_{p,n}\hat M_{p,n}^{*}) + (M_{p,n}'\hat M_{p,n}^{*} + \hat M_{p,n}M_{p,n}'^{*} + M_{p,n}'M_{p,n}'^{*}) \\
\notag &\mathdef \hat \Sigma_{p,n} + \Sigma_{p,n}'.
\end{align*}
Using triangular inequality, we have
\begin{align*}
\lambda_{1}(\Sigma_{p,n}) &= ||\Sigma_{n,p}|| \le ||\hat M_{p,n} + M_{p,n}'||^{2} \le [||\hat M_{p,n}|| + ||M'_{p,n}||]^{2} \\
& \le  ||\hat \Sigma_{p,n}|| + 2  || \hat \Sigma_{p,n}||^{1/2} (||M_{p,n}'||_{1}||M_{p,n}'||_{\infty})^{1/2} + ||M_{p,n}'||_{1}||M_{p,n}'||_{\infty}.
\end{align*}
Hence, we will prove, with probability tending to one,
\begin{align}
\label{tbproved1}
||\hat \Sigma_{p,n}|| &=o(c_{np}^{2})\\
\label{tbproved2}
||M_{p,n}'||_{\infty} \le |m_{i_{1}j_{1}}| (1+o(1)), &\quad ||M_{p,n}'||_{1} \le |m_{i_{1}j_{1}}| (1+o(1))
\end{align}
which gives \eqref{eqn:sampcovub}.
For \eqref{tbproved1}, first, one can deduce from the lower bound of $\gamma$ that $\mu + \gamma(1-\alpha) < \frac{\mu+1}{\alpha}$, and hence,
\begin{align*}
\left | ||\hat M_{p,n} || - || \hat M_{p,n} - \E \hat m_{ij}|| \right| \le \sqrt{np}\, \E \hat m_{ij} \le CL(n^{\gamma}) n^{\mu + \gamma(1-\alpha)} = o(n^{\frac{\mu+1}{\alpha}}) = o(c_{np}).
\end{align*}
So we may assume that the truncated entries are centered. Here, the first inequality is a consequence of \cite[Theorem A.46]{Bai:2009kn} and the second inequality is due to \cite[Lemma 13]{Auffinger:2009vs}. Theorem \ref{thm:traceboundcov} indicates that
\[
\Pro(||\hat \Sigma_{p,n}|| \ge C n^{2\gamma'}) \to 0, 
\]
Now with $\gamma'$ chosen such that $\gamma' < \frac{\mu+1}{\alpha}$, \eqref{tbproved1} holds with probability tending to one.\br 
For \eqref{tbproved2}, again, we only show the upper bound for the infinity norm. As in the previous case, it is enough to show that for any  $1\le i \le n$ fixed, w.e.h.p.,
\[
\tilde S_{i} \mathdef \sum_{j=1}^{n}|m_{ij}|\mathbf 1_{\{n^{\gamma} < |m_{ij}| \le c_{np}^{\kappa}\}} = o(n^{\frac{\mu+1}{\alpha}}).
\]
We treat $\tilde S_{i}$ similarly:
\begin{align*}
\tilde S_{i}  &=  \sum_{j=1}^{n}|m_{ij}|\mathbf 1_{\{n^{\gamma} < |m_{ij}| \le n^{\frac{\mu+1}{\alpha} - \eta}\}} + \sum_{j=1}^{n}|m_{ij}|\mathbf 1_{\{ n^{\frac{\mu+1}{\alpha} - \eta} < |m_{ij}| \le n^{\frac{\mu+1}{\alpha} + \eta}\}  } \\
\qquad &+ \sum_{j=1}^{n}|m_{ij}|\mathbf 1_{\{ n^{\frac{\mu+1}{\alpha} + \eta} \le |m_{ij}| <  c_{np}^{\kappa} \} } \mathdef \tilde S_{i, 1} + \tilde S_{i, 2} + \tilde S_{i, 3}.
\end{align*}
Here, the only difference from the previous case is the $\tilde S_{i,1}$ term. By Lemma \ref{lem:keylemma2} \eqref{lemma3.2b}, for arbitrarily small $\delta > 0$, w.e.h.p.,
$
\tilde S_{i, 1} \le n^{\mu-\gamma(\alpha-1) + \delta}
$. 
The choice of  
$
\gamma > \frac{\mu}{\alpha} - \frac{1}{\alpha(\alpha-1)}
$ guarantees $\tilde S_{i, 1}= o( n^{\frac{\mu+1}{\alpha}})$.\Br
The part applying Cauchy interlacing theorem is identical to Case I, so we remain to show that with probability going to one,  the norm of $\mathbf r_{l}$, as defined in \eqref{Sunnyoutsidebutpreferhere} is of smaller order with respect to $c_{np}^{2}$, as $n \to \infty$. We estimate $||\mathbf r_{l}||$ using the triangular inequality and the decomposition of $\Sigma_{p,n}$ above:
\begin{align*}
||\mathbf r_{l}|| &\le ||\hat \Sigma_{p,n}|| + ||M_{p,n}'\hat M_{p,n}^{*}\mathbf e_{i_{l}}|| +||\hat M_{p,n} M_{p,n}'^{*}\mathbf e_{i_{l}}|| + ||M_{p,n}'M_{p,n}'^{*}\mathbf e_{i_{l}} - |m_{i_{l}j_{l}} |^{2}\mathbf e_{i_{l}} ||\\
&\le ||\hat \Sigma_{p,n}|| + 2  ||\hat \Sigma_{p,n}||^{1/2} (||M_{p,n}'||_{1}||M_{p,n}'||_{\infty})^{1/2}  +  ||M_{p,n}'M_{p,n}'^*\mathbf e_{i_{l}} - |m_{i_{l}j_{l}} |^{2}\mathbf e_{i_{l}} ||.
\end{align*}
% Hence, it suffices to show
% \begin{equation*}
% \Pro \bigg( ||\hat \Sigma_{p,n}|| + 2 ||M_{p,n}'||\cdot ||\hat \Sigma_{p,n}||^{1/2} +  ||M_{p,n}'M_{p,n}'^{*}\mathbf e_{i_{l}} - |m_{i_{l}j_{l}} |^{2}\mathbf e_{i_{l}} ||=o(c^{2}_{np}) \bigg ) \to 1.
% \end{equation*}
In view of \eqref{tbproved1} and \eqref{tbproved2}, 
% that $||\hat \Sigma_{p,n}||$ is $o(c_{np}^{2})$ and 
% $$||M_{p,n}'|| \le \sqrt{||M_{p,n}'||_{1}||M'_{p,n}||_{\infty}} \le |m_{i_{1}j_{1}}|^{2}(1+o(1)), $$
we remain to show that with probability going to one,
\begin{equation}
\label{toprove3}
||M_{p,n}'M_{p,n}'^*\mathbf e_{i_{l}} - |m_{i_{l}j_{l}} |^{2}\mathbf e_{i_{l}} ||=o(c^{2}_{np}). 
\end{equation}
We compute the left side directly, which yields
\begin{align*}
||M_{p,n}'M_{p,n}'^*\mathbf e_{i_{l}} - |m_{i_{l}j_{l}} &|^{2}\mathbf e_{i_{l}} ||\le \sum_{\substack{k=1\\k\ne j_{l}}}^{n} |m_{i_{l}k}|\mathbf 1_{\{|m_{i_{l}k}|>n^{\gamma}\}}\sum_{\substack{s=1\\s\ne i_{l}}}^{p} |m_{sk}|\mathbf 1_{\{|m_{sk}|>n^{\gamma}\}} \\
&+ |m_{i_{l}j_{l}}| \sum_{\substack{s=1\\s\ne i_{l}}}^{p} |m_{sj_{l}}|\mathbf 1_{\{|m_{sj_{l}}|>n^{\gamma}\}} + \left(\sum_{\substack{k=1\\k\ne j_{l}}}^{n}| m_{i_{l}k}|\mathbf 1_{\{|m_{i_{l}k}|>n^{\gamma}\}}\right)^{2}. 
\end{align*}
Using Lemma \ref{lem:keylemma1} and \ref{lem:keylemma2} \eqref{lemma3.2b}, \eqref{lemma3.2c} and \eqref{lemma3.2d}, we see that each summation above is $o(c_{np})$ with probability tending to one and hence \eqref{toprove3} is proved.  We conclude that $||\mathbf r_{l}||=o(c_{np}^{2})$ with probability tending to one and the proof is complete.
\subsubsection{Eigenvectors} $ $\\
We consider the matrix $c_{np}^{-2}\Sigma_{p,n}$. Let $\mathbf v = \mathbf e_{i_{l}}$ and $\zeta = \langle \mathbf v, c_{np}^{-2}\Sigma_{p,n}\mathbf v\rangle = c_{np}^{-2}\sum_{j=1}^{n}|m_{i_{l}j}|^{2}$. Then
\begin{align*}
\epsilon &= c_{np}^{-2}||(\Sigma_{p,n}-\zeta)\mathbf v|| \le c_{np}^{-2}||(\Sigma_{p,n}-|m_{i_{l}j_{l}}|^{2})\mathbf e_{i_{l}}|| + c_{np}^{-2}||(\zeta - |m_{i_{l}j_{l}}|^{2}) \mathbf e_{i_{l}}|||\\
&\le c_{np}^{-2}||\mathbf r_{l}|| + c_{np}^{-2}\sum_{j=1,j\ne j_{l}}^{n}|m_{i_{l}j}|^{2}.
\end{align*}
By Lemma \ref{lem:keylemma1} and \ref{lem:keylemma2}, we know $\epsilon \stackrel{P}{\rightarrow}0$. Hence, in order to use Theorem \ref{thm:pertubeigs}, it suffices to show that for sufficiently small $\delta > 0$, with probability tending to one, $\lambda_{l}$ is the only eigenvalue in $\overline{B(\zeta, \delta)}$, i.e., each $k \ge 1$, the spacing of the eigenvalues satisfies
\[
\lim_{\delta\to 0}\limsup_{n\to\infty}\Pro\left(\frac{\lambda_{k}(\Sigma_{p,n}) - \lambda_{k+1}(\Sigma_{p,n})}{c_{np}^{2}} < \delta \right) = 0. 
\]
Since we have proved $\lambda_{k}(\Sigma_{p,n})/|m_{i_{k}j_{k}}|^{2}\stackrel{P}{\rightarrow} 1$, this is equivalent to
\[
\lim_{\delta\to 0}\limsup_{n\to\infty}\Pro\left(c_{np}^{-2}|m_{i_{k}j_{k}}|^{2} -c_{np}^{-2}|m_{i_{k+1}j_{k+1}}|^{2} < \delta \right) = 0.
\]
However, this follows from the fact that $\mathcal Q_{n} =\sum_{i=1}^{p}\sum_{j=1}^{n} \delta_{c_{np}^{-2}|m_{ij}|^{2}}$ converges to a Poisson point process on $(0, +\infty)$.
\subsection{Proof of Theorem \ref{thm:mainresult4} } 
\subsubsection{Eigenvalues}\label{sec:lastpart}

Proposition \ref{prop:samplecovariance} implies that for any $k\ge 1$ fixed, and any $\epsilon > 0$
\[
\Pro(\lambda_{k}(\Sigma_{p,n}) \ge (1+\sqrt{\rho})^2(1-\epsilon)n^{\mu}) \to 1, \quad \text{as } n\to\infty.
\]
It remains to prove the upper bound, i.e., $\Pro(\lambda_{k}(\Sigma_{p,n}) \leq (1+\sqrt{\rho})^2(1+\epsilon)n^{\mu}) \to 1$. Since $\lambda_{k}(\Sigma_{p,n}) \le \lambda_{1}(\Sigma_{p,n})$ and since we can decompose $\Sigma_{p, n}$ in the same way as in the previous case, it is enough to show that 
\begin{equation*}
\Pro\bigg( ||\hat \Sigma_{p,n}|| + 2  || \hat \Sigma_{p,n}||^{1/2} (||M_{p,n}'||_{1}||M_{p,n}'||_{\infty})^{1/2} + ||M_{p,n}'||_{1}||M_{p,n}'||_{\infty}  >  (1+\sqrt{\rho})^2(1+\epsilon)n^{\mu}\bigg)
\end{equation*} goes to zero.
In this regime, we choose $\gamma' = \mu/2$ and $\gamma\in(\frac{\mu}{2(\alpha-1)}, \frac{\mu}{2})$, which is always possible when $\alpha > 2$. Such $\gamma$ and $\gamma'$ satisfy the assumptions in Theorem \ref{thm:traceboundcov}, which gives the bound for the truncated part, i.e.,
\begin{equation*}
\Pro(||\hat \Sigma_{p,n}|| \ge (1+\sqrt{\rho})^2(1+\epsilon)n^{\mu}) \to 0.
\end{equation*}
We remain to show that 
$
||M_{p, n}'||_{1} = o(n^{\mu/2})
$ and $||M_{p, n}'||_{\infty} = o(n^{\mu/2})$ with probability tending to one. Again, we prove for the infinity norm only, i.e., with probability going to one,
\[
 S_{i} \mathdef \sum_{j=1}^{n}|m_{ij}|\mathbf 1_{\{|m_{ij}|> n^{\gamma}\}} = o(n^{\mu/2}), \quad \text{for all } 1\le i\le n.
\] 
Since $c_{np}\slvar n^{(1+\mu)/\alpha}$ and $c_{np}^{-2}|m_{i_{1}j_{1}}|^{2} $ converges in distribution,  then for any $\theta > \frac{\mu+1}{\alpha}$, with probability tending to one, 
$
\max_{1\le i,j \le n}|m_{ij}| \le n ^{\theta}
$. Hence, with probability tending to one, for all $1\le i\le n$, 
\begin{align*}
S_{i}&=\sum_{j=1}^{n}|m_{ij}|\mathbf 1_{\{n^{\gamma} < |m_{ij}|\le n^{\frac{\mu}{\alpha}}\}}+\sum_{j=1}^{n}|m_{ij}|\mathbf 1_{\{n^{\frac{\mu}{\alpha}} < |m_{ij}|\le n^{\frac{\mu+1}{\alpha}}\}}+\sum_{j=1}^{n}|m_{ij}|\mathbf 1_{\{n^{\frac{\mu+1}{\alpha}} < |m_{ij}|\le n^{\theta}\}}\\
 &\mathdef S_{i,1} + S_{i,2}+S_{i,3}.
 \end{align*}
By Lemma \ref{lem:keylemma2} \eqref{lemma3.2b}, \eqref{lemma3.2c} and \eqref{lemma3.2d}, respectively, for any $\epsilon > 0$, we have w.e.h.p., 
\begin{align*}
S_{i,1} &\le n^{\mu-\gamma(\alpha-1)+\epsilon}, \quad S_{i,2} \le n^{\frac{\mu+1}{\alpha}+\epsilon},\quad S_{i,3} \le n^{\theta+\epsilon}.
\end{align*}
As $\alpha > 2(1+\mu^{-1})$ we have $\frac{\mu+1}{\alpha} < \frac{\mu}{2}$. We can choose $\theta$ arbitrarily close to $\frac{\mu+1}{\alpha}$ and $\epsilon > 0$ small enough such that w.e.h.p., $S_{i,2} + S_{i,3} =o(n^{\mu/2})$. Moreover, the choice of $\gamma > \frac{\mu}{2(\alpha-1)}$ implies  $\mu-\gamma(\alpha-1) < \frac{\mu}{2}$. By making $\epsilon$ small enough, we get w.e.h.p., $S_{i,1} = o(n^{\mu/2})$. Therefore, with probability tending to one, for all $1\le i \le n$, $S_{i} = o(n^{\mu/2})$ hence $||M_{p,n}'||_{\infty} = o(n^{\mu/2})$. The proof is complete.

\subsubsection{Eigenvectors} Last, we prove the localization result of the eigenvectors of $\Sigma_n$. We use the following simple linear algebra lemma \cite[Lemma 4.2]{BenaychGeorges:2014ki} that we quote without proof. 

\begin{lem} \label{lem:linearalgebraeasylemma} Let $H$ be a Hermitian matrix and $\rho_L(H)$ be the maximum spectral radius of its $L \times L$ principal sub-matrix. Let $\lambda$ be an eigenvalue of $H$ and $\mathbf v$ an associated unit eigenvector.     
If $\mathbf v$ is $(L, \eta)$-localized, then  $$|\lambda| \leq \frac{\rho_L(H)+\sqrt{\eta} \|H \|}{\sqrt{1-\eta}}.$$
\end{lem}

If $\alpha > 2(1+\mu^{-1})$ then, by Theorem \ref{thm:traceboundcov}, we know that $\| \Sigma_{p, n} \|$ is of order $(1+\sqrt{\rho})^2 n^{\mu}$. In view of Lemma \ref{lem:linearalgebraeasylemma}, it suffices to show that there exists $\eta>0$ such that with probability going to one 
\begin{equation*}\label{eq:evlocal}
\rho_{\lfloor p^\beta \rfloor}(\Sigma_{p,n})<(\sqrt{1-\eta}-\sqrt{\eta})(1+\sqrt{\rho})^2n^{\mu}.
\end{equation*}
In other words, we must establish that  there exist $\epsilon >0$ such that with probability going to one any $\lfloor p^\beta \rfloor \times \lfloor p^\beta  \rfloor$ principal sub-matrix $W$ of $\Sigma_{n,p}$ satisfies  $\|W\| \leq (1+\epsilon)n^{\mu}$.

We proceed as follows. A principal sub-matrix $W$ is obtained by choosing $\lfloor p^\beta \rfloor$ rows $i_1, \ldots, i_{\lfloor p^\beta \rfloor}$ of the rectangular matrix $M_{p,n}$ and writing $W=M_IM^{*}_I$, where $I = \{i_1, \ldots, i_{\lfloor p^\beta \rfloor} \}$. Here, the notation $M_I$ stands for the $\lfloor p^\beta \rfloor \times n$ sub-matrix of $M$ formed by the rows with indices in $I$. As before, we write $\hat M$ and $M'$ for the truncation and remainder of the matrix $M$ at level $n^{\gamma}$, for $\gamma \in (\frac{\mu}{2(\alpha-1)}, \frac{\mu}{2})$. Then  $\|W\| \leq \| \hat M_I \hat M_I^{*} \| + 2   \| \hat M_I \hat M_I^{*} \|^{1/2} (||M_{I}'||_{1}||M_{I}'||_{\infty})^{1/2}+ \| M_I' M_I'^*\|$. \br
For any choice of $I$, $\| M_I' M_I'^* \| \leq \| M_{p,n}'\|_1 \|M_{p,n}'\|_\infty = o(n^{\mu})$ as in the proof in Section \ref{sec:lastpart}. On the other hand, for any choice of $I$, one can adapt the proof of Theorem \ref{thm:traceboundcov} to deal with the case of $\hat p = \lfloor p^\beta\rfloor$ rows to show that for any $1<c<(1+\sqrt{\rho})^{2},$ there exists $\theta=\theta(c)>0$ and $\gamma'=\gamma'(c)>0$ so that 
$$\Pro \left( \|\hat M_I \hat M_I^{*} \| \geq c n^{\mu}\right) \leq P(n)n^{-\theta\lfloor n^{\gamma'} \rfloor},$$
where $P$ is a polynomial in $n$. Indeed, in the case where $\hat p \to \infty$ and $\hat p/n \to 0$ one needs to control the appearance of odd and even innovations (or odd/even marked vertices as in \cite[Section 2.2]{Peche}). We leave the details to the reader. Since there are at most $n^{p^\beta}$ ways to choose the indices in $I$, the probability of the existence of such a principal sub-matrix is bounded above by:
$$ P(n)n^{-\theta \lfloor n^{\gamma'} \rfloor } n^{(2\rho n)^\beta}.$$
Thus if we choose $\beta<\gamma'$, we obtain the desired result.

\section{Hermitian Matrices }\label{sec:extension} In the last section, we derive one extension of the methods of Sections \ref{sec3} and \ref{sec4}. 
\subsection{Sparse Hermitian matrices with heavy tails. }  Recall that $x$ is a random variable with heavy tailed distribution and $y$ is a Bernoulli random variable, independent of $x$, with success probability $n^{\mu-1}$. Let $X_{n} = [x_{ij}]_{i,j}^n$ be an $n \times n$ Hermitian, random matrix where entries along and above the diagonal are i.i.d. copies of $x$ and $Y_{n}=[y_{ij}]_{i,j=1}^{n}$ be a real, symmetric matrix whose entries along and above the diagonal are i.i.d. copies of $y$.
Define 
\begin{equation*}
\label{eqn:defMn}
M_{n}=X_{n}\cdot Y_{n}=[m_{ij}]_{i,j=1}^{n}\mathdef [x_{ij}y_{ij}]_{i,j=1}^{n}.
\end{equation*} 
Since there are, on average, $\frac{n^{2}\cdot n^{\mu-1}}{2} = \frac{n^{\mu +1 }}{2}$ independent, nonzero entries in $M_{n}$, the right scaling factor for the largest entries of $M_{n}$ should be
\begin{equation*}
\label{eqn:cn}
c_{n} \mathdef  
\inf\left\{t: G_{\alpha}(t) \le \frac{2}{(n+1)n^{\mu}}\right\} \slvar n^{\frac{\mu+1}{\alpha}}.
\end{equation*}

Using the a similar argument as in Sections \ref{sec3} and \ref{sec4}, one can prove the different behavior for the eigenvalues and eigenvectors of $M_{n}$, determined by the tail exponent $\alpha$ and the sparsity exponent $\mu$. When $0<\alpha < 2(1+\mu^{-1})$, the largest eigenvalues of $M_{n}$  behave like its largest entries. In this case, $n^{\frac{\mu+1}{\alpha}} \gg n^{\frac{\mu}{2}}$, where $n^{\frac{\mu}{2}}$ is the order of the bulk of the spectrum. The eigenvectors are localized.

\begin{thm}\label{thm:mainresult1} Suppose $0 < \alpha < 2(1+\mu^{-1})$.  For $(1+\mu^{-1}) \le \alpha < 2(1+\mu^{-1})$, we also assume that $x$ is centered. Then for each $l\ge 1$,  as $n\to \infty$, we have
\begin{equation}
\label{eqn:eigsub}
\frac{\lambda_{l}(M_{n})}{|m_{i_{l}j_{l}}|} \stackrel{P}{\longrightarrow} 1,
\end{equation}
and the localization of the corresponding eigenvector, i.e.,
\begin{equation*}
\left \Vert \mathbf v_{l}(M_{n}) - \frac{1}{\sqrt{2}}(e^{i\theta_{l}(M_{n})/2}\mathbf e_{i_{l}} +  e^{-i\theta_{l}(M_{n})/2}\mathbf e_{j_{l}})\right \Vert_{2} \stackrel{P}{\longrightarrow} 0.
\end{equation*}
\end{thm}\vspace{4mm}

\noindent Again, \eqref{eqn:eigsub} implies  that the random point process
$
\hat{\mathcal Q}_{n} =\sum_{i=1}^{n} \delta_{c_{n}^{-1}\lambda_{i}(M_{n})}\mathbf 1_{\lambda_{i}(M_{n})>0}
$
also converges in distribution to a Poisson point process on $(0, + \infty)$ with intensity $\frac{\alpha}{x^{1+\alpha}}$ as $n\to \infty$.

When $\alpha> 2(1+\mu^{-1})$, the analogue of Theorem \ref{thm:mainresult4} is the following.
\begin{thm}\label{thm:mainresult2} Suppose $ \alpha > 2(1+\mu^{-1})$ and  that $x$ has mean zero and variance one. Then for each $l\ge 1$, as $n\to \infty$, we have
\[
\frac{\lambda_{l}(M_{n})}{n^{\mu/2}} \stackrel{P}{\longrightarrow} 2.
\]
The eigenvectors are delocalized: there exist $\beta>0$ such that for each $l\geq 1$, $\eta < 1/2$
\begin{equation*}\label{eq:delocthm}
\Pro\left(\exists l: |\lambda_{l}(M_{n})|\ge \sqrt{2\eta} ||M_{n}|| \text{ and } \mathbf v_{l}(M_{n}) \text{ is } (\lfloor n^{\beta}\rfloor, \eta)\text{-localized}\right) \to 0
\end{equation*}
as $n\to \infty$.
\end{thm}

Under the extra assumptions that the matrix $M_{n}$ has, asymptotically, a fixed number $n^{\mu}$ nonzero entries in almost all rows and no randomness in their positions, 
Theorems \ref{thm:mainresult1} and \ref{thm:mainresult2} were proved in \cite{BenaychGeorges:2014ki}. It is not difficult to modify the arguments there to include the above results. For instance, Lemmas \ref{lem:keylemma1} and \ref{lem:keylemma2} still hold if we replace $M_{p,n}$ by the Hermitian matrix $M_{n}$ and $c_{np}$ by $c_{n}$, and this modification is well suited for deriving an upper bound for the infinity norms. The next proposition is a small modification of \cite[Theorem 2.1]{BenaychGeorges:2014ki}, which allows us to deal with the fact that here the number of nonzero elements in a row is random and not bounded by $n^\mu$.
\begin{prop}  
\label{thm:tracebound}  Suppose $\alpha > 2$ and $x$ has mean zero and variance one. Let $M_{n}=X_{n}\cdot Y_{n}$ be the sparse $n \times n$ Hermitian matrix with heavy tailed entries. Consider positive exponents $\gamma, \gamma'$ and $\gamma''$ such that $\frac{\mu}{2} \le \gamma'$ and $\frac{\mu}{4}+\gamma + \gamma'' < \gamma'$. Define the truncated matrix 
\[
\hat M_{n}  = [\hat m_{ij}]_{i, j=1}^{n}= [m_{ij}\mathbf 1_{\{|m_{ij}| \le n^{\gamma}\}}]_{i,j=1}^{n}.
\]
We also assume that the truncated entries are centered. Then for $s_{n}\le n^{\gamma''}$, there exists a slowly varying function $L_{0}$ such that for any $C > 0$
\[
\E [\Tr(\hat M_{n}^{2s_{n}})\mathbf 1_{\{\mathsf  L \le Cn^{\mu}\} }] \le L_{0}(n)n^{1+2\gamma} s_{n}^{-3/2}(2n^{\gamma'})^{2s_{n}},
\]
where $\mathsf L \mathdef \max_{1\le i \le n}\sum_{j=1}^{n}\mathbf 1_{\{|m_{ij}|> 0\}} = \max_{1\le i\le n}\sum_{j=1}y_{ij}$, i.e., the maximum of the number of nonzero entries of all rows.
\end{prop}

\begin{proof} First, note that it is equivalent to truncate the $X_{n}$ matrix, i.e.,
\[
\hat M_{n} = [\hat m_{ij}]_{i,j=1}^{n} = [x_{ij}y_{ij}\mathbf 1_{\{|x_{ij}y_{ij}|\le n^{\gamma}\}}]_{i,j=1}^{n} = [x_{ij}\mathbf 1_{\{|x_{ij}|\le n^{\gamma}\}}\cdot  y_{ij}]_{i,j=1}^{n} \mathdef [\hat x_{ij} y_{ij}]_{i,j=1}^{n}.
\]

By independence, we write the expected trace on the event $\{\mathsf L \le Cn^{\mu}\}$
\begin{align*}
\notag
\E [\Tr(\hat M_{n}^{2s_{n}})\mathbf 1_{\{ L \le Cn^{\mu} \} }] & = \sum_{i_{1},\ldots, i_{2s_{n}}=1}^{n} \E (\hat m_{i_{1}i_{2}}\cdots \hat m_{i_{2s_{n}}i_{1}}\mathbf 1_{\{ \mathsf L \le Cn^{\mu} \} } )\\
&
=\sum_{i_{1},\ldots, i_{2s_{n}}=1}^{n} \E (\hat x_{i_{1}i_{2}}\cdots \hat x_{i_{2s_{n}}i_{1}} )\E( y_{i_{1}i_{2}}\cdots y_{i_{2s_{n}}i_{1}} \mathbf 1_{\{\mathsf L \le Cn^{\mu} \} } ).
\end{align*}
Once this factorization is written, the proof follows immediately from the same combinatorics presented in the proof of \cite[Theorem 2.1]{BenaychGeorges:2014ki}. Note $ y_{i_{1}i_{2}}\cdots y_{i_{2s_{n}}i_{1}} \mathbf 1_{\{\mathsf L \le Cn^{\mu} \} }$ is nonzero only if all rows have no more than $Cn^{\mu}$ nonzero entries. When labeling the vertices in a path $\mathbf i = (i_{1}, \ldots, i_{2s_{n}})$, we have $Cn^{\mu}$ possible choices for each vertex instead of $n^{\mu}$ in \cite[Theorem 2.1]{BenaychGeorges:2014ki}. However, the extra factor $C$ will not play a role in determining the upper bound. We omit the lengthy calculation here.
\end{proof}
%The proof of Theorems  \ref{thm:mainresult1} and \ref{thm:mainresult2} now follows from the analogues of Lemmas \ref{lem:keylemma1} and \ref{lem:keylemma2}, Proposition \ref{thm:tracebound}, and Chernoff's bound. We leave the details to the reader.
%

\bibliographystyle{plain}

\begin{thebibliography}{10}

\bibitem{Anderson:2010up}
Greg~W. Anderson, Alice Guionnet, and Ofer Zeitouni, \emph{{An Introduction to
  Random Matrices}}, Cambridge University Press, 2010.

\bibitem{Auffinger:2009vs}
Antonio Auffinger, G\' erard Ben~Arous, and Sandrine P{\'e}ch{\'e}, \emph{{Poisson
  convergence for the largest eigenvalues of Heavy Tailed Random Matrices}},
  Annales de lInstitut Henri Poincare \textbf{45} (2009), 589--610.

\bibitem{Bai:2009kn}
Zhidong Bai and Jack~W. Silverstein, \emph{{Spectral Analysis of Large Dimensional Random Matrices}}, Springer Science {\&} Business Media, 2009.

\bibitem{Bai1988166}
Zhidong Bai, Jack~W. Silverstein, and Y.~Q. Yin, \emph{A note on the largest
  eigenvalue of a large dimensional sample covariance matrix}, Journal of
  Multivariate Analysis \textbf{26} (1988), no.~2, 166--168.

\bibitem{Bai:1988je}
Zhidong Bai and Y.~Q. Yin, \emph{{Necessary and sufficient conditions for almost
  sure convergence of the largest eigenvalue of a Wigner matrix}}, The Annals
  of Probability (1988), 1729--1741.

\bibitem{BenaychGeorges:2014gg}
Florent Benaych-Georges and Alice Guionnet, \emph{{Central limit theorem for
  eigenvectors of heavy tailed matrices}}, Electronic Journal of Probability
  \textbf{19} (2014), 1--27.

\bibitem{BenaychGeorges:2014ki}
Florent Benaych-Georges and Sandrine P{\'e}ch{\'e}, \emph{{Localization and
  delocalization for heavy tailed band matrices}}, Annales de l'Institut Henri
  Poincar{\'e}, Probabilit{\'e}s et Statistiques \textbf{50} (2014), no.~4,
  1385--1403.

\bibitem{ECP3027}
Florent Benaych-Georges and Sandrine P\'ech\'e, \emph{Largest eigenvalues and
  eigenvectors of band or sparse random matrices}, Electron. Commun. Probab.
  \textbf{19} (2014), no. 4, 1--9.

\bibitem{Bennett:1962ev}
George Bennett, \emph{{Probability Inequalities for the Sum of Independent
  Random Variables}}, Journal of the American Statistical Association
  \textbf{57} (1962), no.~297, 33--45.

\bibitem{Regular}
Nicholas H. Bingham, Charles M. Goldie and Jozef. L. Teugels,  \emph{Regular variation}, Cambridge University Press, 1987.

\bibitem{Bhatia:1997tf}
Rajendra Bhatia, \emph{{Matrix analysis}}, Graduate texts in mathematics.
  Springer, 1997.

\bibitem{Bourgade:2013wd}
Paul Bourgade and Horng-Tzer Yau, \emph{{The Eigenvector Moment Flow and local
  Quantum Unique Ergodicty}}, arXiv:1312.1301 (2013).

\bibitem{CB94}
Pierre Cizeau and Jean-Philippe Bouchaud, \emph{Theory of L\'evy matrices}, Phys. Rev. E
  \textbf{50} (1994), 1810--1822.

\bibitem{Erdos:2012cx}
L{\'a}szl{\'o} Erd{\H o}s, Antti Knowles, Horng-Tzer Yau, and Jun Yin,
  \emph{{Spectral statistics of Erd{\H o}s-R\'enyi Graphs II: Eigenvalue
  spacing and the extreme eigenvalues}}, Communications in Mathematical Physics
  \textbf{314} (2012), no.~3, 587--640.

\bibitem{Feller:1971ve}
William Feller, \emph{{An introduction to probability theory and its
  applications, Vol. II}}, 1971.

\bibitem{Khor08}
Oleksiy Khorunzhiy, \emph{Estimates for moments of random matrices with
  gaussian elements}, Lecture Notes in Mathematics \textbf{1934} (2008), 51--92
  (English).

\bibitem{lee2014}
Ji~Oon Lee and Jun Yin, \emph{A necessary and sufficient condition for edge
  universality of wigner matrices}, Duke Math. J. \textbf{163} (2014), no.~1,
  117--173.

\bibitem{540625382}
Alexander Mirlin and Yan~V. Fyodorov, \emph{On the density of states of sparse
  random matrices}, J. Phys. A: Math. Gen. \textbf{24} (1991), no.~10, 2219.

\bibitem{Nagao13}
Taro Nagao, \emph{Spectral density of complex networks with two species of
  nodes}, Journal of Physics A: Mathematical and Theoretical \textbf{46}
  (2013), no.~6, 065003.

\bibitem{1751-8121-40-19-003}
Taro Nagao and Toshiyuki Tanaka, \emph{Spectral density of sparse sample
  covariance matrices}, Journal of Physics A: Mathematical and Theoretical
  \textbf{40} (2007), no.~19, 4973.

\bibitem{Peche}
Sandrine P\'ech\'e, \emph{Universality results for the largest eigenvalues of some sample covariance matrix ensembles},
Probab. Theory Relat. Fields \textbf{143}, (2009) 481--516.


\bibitem{Resnick:2007uq}
Sidney~I. Resnick, \emph{{Extreme Values, Regular Variation, and Point
  Processes}}, Springer Science {\&} Business Media, November 2007.

\bibitem{540625391}
Geoff Rodgers and Cyrano De~Dominicis, \emph{Density of states of sparse random
  matrices}, J. Phys. A: Math. Gen. \textbf{23} (1990), no.~9, 1567.

\bibitem{SemerC02}
Guilhem Semerjian and Leticia~F Cugliandolo, \emph{Sparse random matrices: the
  eigenvalue spectrum revisited}, Journal of Physics A: Mathematical and
  General \textbf{35} (2002), no.~23, 4837.

\bibitem{Sodin}
Sasha Sodin, \emph{The spectral edge of some random band matrices}, Ann. of
  Math (2) \textbf{172} (2010), 2223--2251 (English).

\bibitem{Soshnikov:2004uc}
Alexander Soshnikov, \emph{{Poisson statistics for the largest eigenvalues of
  Wigner random matrices with heavy tails}}, Electronic Communications in
  Probability \textbf{9} (2004), 82--91.

\bibitem{Tao:2012ix}
Terence Tao and Van Vu, \emph{{Random covariance matrices: universality of
  local statistics of eigenvalues}}, The Annals of Probability \textbf{40}
  (2012), no.~3, 1285--1315.

\bibitem{Yin:1988kp}
Y.~Q. Yin, Zhidong Bai, and Pathak R. Krishnaiah, \emph{{On the limit of the largest
  eigenvalue of the large-dimensional sample covariance matrix}}, Probability
  Theory and Related Fields \textbf{78} (1988), no.~4, 509--521.

\bibitem{Yoshida}
Mika Yoshida and Toshiyuki Tanaka, \emph{{Analysis of Sparsely-Spread CDMA via Statistical
  Mechanics}}, Proc. 2006 IEEE Int. Symp. Info. Theory (2006), 2278--2382.

\end{thebibliography}

\end{document}